\documentclass[article]{amsart}

\usepackage{amssymb,amsfonts,amsmath,amsthm}
\usepackage[all,arc]{xy}
\usepackage{enumerate}
\usepackage{mathrsfs}
\usepackage[toc,page]{appendix}
\usepackage{graphicx}
\usepackage[font=small,labelfont=bf]{caption}
\usepackage[left=3cm, right=3cm, bottom=3cm]{geometry}
\usepackage{tabularx}
\usepackage{url}
\usepackage{color}
\usepackage{tikz-cd}

\newtheorem{thm}{Theorem}[section]

\newtheorem*{thm*}{Theorem}
\newtheorem{cor}[thm]{Corollary}
\newtheorem{prop}[thm]{Proposition}
\newtheorem*{prop*}{Proposition}
\newtheorem{lem}[thm]{Lemma}

\theoremstyle{definition}
\newtheorem{defn}[thm]{Definition}

\newtheorem{exmp}[thm]{Example}

\newtheorem{notn}[thm]{Notation}

\theoremstyle{remark}
\newtheorem{rem}[thm]{Remark}

\newtheorem*{idea*}{Idea}

\makeatletter
\let\c@equation\c@thm
\makeatother
\numberwithin{thm}{section}
\numberwithin{equation}{section}

\newcommand{\SL}{\text{SL}}

\newcommand{\tot}{\mathrm{tot}}

\title{\textsc{Monodromy of rank 2 parabolic Hitchin systems}}

\author{Georgios Kydonakis, Hao Sun and Lutian Zhao}

\begin{document}
\maketitle
\flushbottom

\begin{abstract}
We study the monodromy of the Hitchin fibration for moduli spaces of parabolic $G$-Higgs bundles in the cases when $G=\text{SL}(2,\mathbb{R})$, $\text{GL}(2,\mathbb{R})$ and $\text{PGL}(2,\mathbb{R})$. A calculation of the orbits of the monodromy with $\mathbb{Z}_{2}$-coefficients provides an exact count of the components of the moduli spaces for these groups.
\end{abstract}

\section{Introduction}

The monodromy of (non-parabolic) Hitchin systems has been closely studied in the literature over the last two decades. J. Copeland in \cite{Cope} used a combinatorial approach to determine the monodromy representation of the Hitchin fibration over hyperelliptic curves; this study involves the $\text{SL}(2,\mathbb{C})$-Hitchin map only over the generic points. Later on, L. Schaposnik in \cite{Schap} expanded on Copeland's approach over any compact Riemann surface using the study of the topology of strata of the moduli space of quadratic differentials from \cite{Walk}. In particular, the monodromy group on the $2$-torsion points of the first cohomology of the fibers of the $\text{SL}(2,\mathbb{C})$-Hitchin fibration is explicitly described in \cite{Schap} by a group of matrices and, subsequently, a study of the orbits of the monodromy group then provides a calculation of the number of connected components of the moduli space of semistable $\text{SL}(2,\mathbb{R})$-Higgs bundles. More recently, D. Baraglia and L. Schaposnik in \cite{BarScha} computed the monodromy for rank 2 twisted $G$-Hitchin systems using lattices of the spectral covering of the underlying Riemann surface. In this paper, we develop the necessary machinery to describe the monodromy in the case of moduli spaces of rank 2 $G$-Higgs bundles equipped with a parabolic structure.

An important remark is that in the non-parabolic case the Hitchin fibration can be described by the Jacobians of the spectral covering, a property coming from the well-known Beauville-Narasimhan-Ramanan (BNR) correspondence \cite{BNR}, which is a one-to-one correspondence in the non-parabolic case. We note,  however, that the parabolic version of the correspondence does not involve in general a one-to-one correspondence and is closely related to the given parabolic structure. Special attention is thus needed when describing the parabolic Hitchin fibration for parabolic $G$-Higgs bundles, when $G$ is a complex Lie group or a real form.

Let $X$ be a smooth Riemann surface of genus $g$ and let $D$ be a reduced effective divisor on $X$. In this paper, we are not working on parabolic $G$-Higgs bundles with arbitrary parabolic structures. We are interested in parabolic bundles, of which the weights in the parabolic structure can be written as a fraction (not necessarily reduced) with denominator $2$. Denote by $\mathcal{M}_{par}(G_{c})$ the moduli space of polystable rank $2$ parabolic $G_{c}$-Higgs bundles over the pair $(X, D)$  such that the weights can be written as a fraction with denominator $2$ (see Notation \ref{212.1}), for the complex Lie groups $G_{c}=\text{SL}(2,\mathbb{C})$, $\text{GL}(2,\mathbb{C})$ and $\text{PGL}(2,\mathbb{C})$. The parabolic Hitchin fibration sends a point from the moduli space $\mathcal{M}_{par}(G_{c})$ to a pair of sections in $\mathcal{H}=\bigoplus_{i=1}^2 H^0(X,(K(D))^i)$. For an element in the regular locus $\eta \in H_{reg} \subseteq \mathcal{H}$ the spectral curve construction involves a double cover $X_\eta \to X$ of the Riemann surface $X$, which in the case of strongly parabolic Higgs bundles is completely ramified over the points in $D$.

We use the fundamental correspondence between parabolic Higgs bundles over $(X, D)$ and Higgs $V$-bundles over the $V$-surface  $M$ (as an orbifold) to develop the parabolic version of the BNR correspondence. This way, we may focus on parabolic line bundles over the spectral curve $X_\eta$ described by the $V$-Picard group $\text{Pic}_V(M_\eta)$ of the corresponding spectral covering $M_{\eta} \to M$ of the $V$-surface $M$. In general, the BNR correspondence for Higgs $V$-bundles (alternatively for parabolic Higgs bundles) is not one-to-one. In \S 3.3 we construct a subvariety of  $\text{Pic}_V(M_\eta)$ over which the BNR correspondence is in fact one-to-one. This construction involves restricting to specific weight type parabolic Higgs bundles over $(X, D)$ and is used in order to understand the 2-torsion points in these subvarieties.

On the other hand, the correspondence to Higgs $V$-bundles opens the way to define appropriate topological invariants for the moduli spaces $\mathcal{M}_{par}(G)$ (see Definition \ref{216}), when $G=\text{SL}(2,\mathbb{R})$, $\text{GL}(2,\mathbb{R})$ and $\text{PGL}(2,\mathbb{R})$, thus providing a minimum number of the connected components of $\mathcal{M}_{par}(G)$. Namely, we summarize our results as follows:

\begin{prop}[\textbf{Propositions 4.1, 4.2 and 4.4}]
Let $X$ be a smooth Riemann surface of genus $g$ and let $D$ be a reduced effective divisor of $s$ many points on $X$.
The minimum number of connected components of the moduli space $\mathcal{M}_{par}(G)$ of polystable parabolic $G$-Higgs bundles over the pair $(X, D)$ is given as follows:
\begin{enumerate}
\item if $G=\text{SL}\left( 2,\mathbb{R} \right)$, the minimum number is $2^{2g+s}+2^s(2g-3+s)$;
\item if $G=\text{GL}\left( 2,\mathbb{R} \right)$, the minimum number is $2^s(2^{2g+s-1}-1)+2^{s}\cdot (g-1+\frac{s}{2})+2^{2g+s-1}$;
\item if $G=\text{PGL}\left( 2,\mathbb{R} \right)$, the minimum number is $2^{2g+s}+2^s(2g-3+s)$.
\end{enumerate}
\end{prop}

In order to improve this minimum component count to an exact component count, we follow a different approach than the Morse-theoretic techniques used in \cite{KSZ2} for solving a similar problem. The methods for the study of the topology of moduli spaces of parabolic $G$-Higgs bundles for $G$ a real Lie group first appeared in the dissertation of M. Logares \cite{Loga} (see also the articles \cite{GaLoMu} and \cite{LogaBetti}), who has studied the case $G = \text{U}(p,q)$; these methods involve the analysis of the moment map of the moduli space, which is a Morse-Bott function.

We rather calculate here the orbits of the monodromy action on the 2-torsion points of lattices over the spectral covering $M_\eta$ of the corresponding $V$-surface $M$, inspired by \cite{BarScha} and \cite{Schap}, where moduli spaces of non-parabolic rank 2 Higgs bundles have been considered. We define the lattice $\Lambda_M$ to be the first $V$-cohomology of the $V$-surface $H^1_V(M,\mathbb{Z})$ and similarly $\Lambda_{M_\eta}:=H^1_V(M_\eta,\mathbb{Z})$ denotes the lattice of the spectral covering of $M$. Let $\Lambda_{P_V}$ be the kernel of the natural map $\Lambda_{M_\eta} \rightarrow \Lambda_M$ and denote by $\widetilde{\Lambda}_{P_V}$ the $\mathbb{Z}_2$-extension of $\Lambda_{P_V}$. We are interested in  the monodromy action on the $2$-torsion points of the lattices $\Lambda_{M}$, $\Lambda_{M_\eta}$ and $\widetilde{\Lambda}_{P_V}$. The reason for this is the following. If $|M|$ is the underlying surface of the $V$-surface $M$ and $D$ is the fixed divisor, then there is a natural surjection $\pi_1(|M|\backslash D) \rightarrow \pi_1(M)$, inducing a map in cohomology $H^1_V(M,\mathbb{Z}) \longrightarrow H^1(|M| \backslash D,\mathbb{Z})$. Now the $2$-torsion points of $\Lambda_M$ is isomorphic to the $2$-torsion points of the lattice $\Lambda_{|M|\backslash D}$, hence we can consider $\Lambda_M$ as the lattice for the noncompact surface $|M| \backslash D$ in this special case. For $\Lambda_{M}[2]$, $\Lambda_{M_\eta}[2]$ and $\widetilde{\Lambda}_{P_V}[2]$ denoting the $2$-torsion points of the lattices, we show the following:

\begin{prop}[\textbf{Propositions 6.1, 7.1 and 7.3}]
The fiber of the parabolic Hitchin fibration ${{\mathsf{\mathcal{M}}}_{par}}\left( G \right)\to \mathsf{\mathcal{H}}$ with respect to a point $a_0$ in the regular locus $H_{reg}$ is given as follows:
\begin{enumerate}
\item if $G=\text{SL}(2,\mathbb{R})$, the fiber is the space $\widetilde{\Lambda}_{P_V}[2]$;
\item if $G=\text{GL}(2,\mathbb{R})$, the fiber is the space $\Lambda_{M_\eta}[2]$;
\item if $G=\text{PGL}(2,\mathbb{R})$, the fiber is the space $\left(\Lambda^{0}_{M_\eta}[2]\bigoplus\Lambda^{\frac{1}{2}}_{M_\eta}[2]\right) / \pi^* \Lambda_M[2],$ where $\Lambda^{0}_{M_\eta}[2] \cong \Lambda^{\frac{1}{2}}_{M_\eta}[2] \cong \Lambda_{M_\eta}[2]$ and the superscript $0$ and $\frac{1}{2}$ denotes the parabolic degree of the corresponding parabolic line bundle.
\end{enumerate}
\end{prop}

Using the description of the fiber of the Hitchin fibration, we next study the monodromy action on the fiber based on particular decompositions of $\Lambda_{M_\eta}[2]$ and $\widetilde{\Lambda}_{P_V}[2]$, which allow us to identify elements in the orbits of the monodromy action; we have accordingly:

\begin{prop}[\textbf{Propositions 6.2 and 7.4}]
The number of orbits under the monodromy action on the Hitchin fibration ${{\mathsf{\mathcal{M}}}_{par}}\left( G \right)\to \mathsf{\mathcal{H}}$ is given as follows:
\begin{enumerate}
\item if $G=\text{SL}(2,\mathbb{R})$, the number of orbits is $2^{2g+s}+2^s(2g-3+s)$;
\item if $G=\text{GL}(2,\mathbb{R})$, the number of orbits is $2^{s}(2^{2g+s-1}-1)+ 2^{s} \cdot (g-1+\frac{s}{2})+ 2^{2g+s-1}$;
\item if $G=\text{PGL}(2,\mathbb{R})$, the number of orbits is $2^{2g+s}+2^s(2g-3+s)$.
\end{enumerate}
\end{prop}

The above propositions now imply our main theorem:

\begin{thm}[\textbf{Theorems 6.4, 7.5 and 7.6}]
Let $X$ be a smooth Riemann surface of genus $g$ and let $D$ be a reduced effective divisor of $s$ many points on $X$. The number of connected components of the moduli space $\mathcal{M}_{par}(G)$ of polystable parabolic $G$-Higgs bundles over the pair $(X, D)$ is
\begin{enumerate}
\item $2^{2g+s}+2^s(2g-3+s)$, if $G=\text{SL}(2,\mathbb{R})$;
\item $2^{s}(2^{2g+s-1}-1)+ 2^{s} \cdot (g-1+\frac{s}{2})+ 2^{2g+s-1}$, if $G=\text{GL}(2,\mathbb{R})$;
\item $2^{2g+s}+2^s(2g-3+s)$, if $G=\text{PGL}(2,\mathbb{R})$.
\end{enumerate}
\end{thm}

In \S 2 we review the main definitions for the moduli spaces which are studied in this article. In \S 3 we consider the Hitchin fibration and the construction of the spectral curve for rank 2 parabolic Hitchin systems. We also discuss here the parabolic version of the BNR correspondence with particular focus on the subvarieties of the Picard group restricted to which the correspondence is one-to-one, as well as on the Prym variety of the spectral covering of the $V$-surface. In \S 4 we include the discussion for the topological invariants leading to the minimum number of connected components, while in \S 5 we study the monodromy action on 2-torsion points on the lattices. Finally, \S 6 and 7 include the exact calculation of the number of orbits of the monodromy action.

One of the main tools used in this article is the BNR correspondence for the rank 2 parabolic Hitchin systems of interest. We have included in an appendix a proof of the correspondence for smooth Deligne-Mumford stacks; this generalization is of independent interest and may be used for broader applications in the context of orbicurves.

\section{Parabolic rank 2 Higgs bundles}
In this preliminary section, we introduce terminology for the moduli spaces of parabolic $G$-Higgs bundles that we are primarily interested in this article. A general theory of parabolic $G$-Higgs bundles in the case of any real reductive Lie group $G$ can be found in \cite{BiGaRi}. Notation \ref{212.1} and Definition \ref{216} give the definitions of the moduli spaces, which will be studied in detail from \S 4 to \S 7.

Let $X$ denote a smooth compact connected Riemann surface of genus $g$ and $D=\left\{ {{x}_{1}},\ldots ,{{x}_{s}} \right\}$ a reduced effective divisor on $X$. Fix this pair $\left( X,D \right)$. Let also $K$ denote the canonical line bundle over $X$; we shall write $K\left( D \right):=K\otimes {{\mathsf{\mathcal{O}}}_{X}}\left( D \right)$. We define next a rank 2 parabolic Higgs bundle over $\left( X,D \right)$ following C. Simpson \cite{Simp}.

\begin{defn}\label{201} Let $E$ be a holomorphic rank 2 vector bundle on $X$. We say $E$ is \emph{parabolic} over $\left( X,D \right)$
if it is endowed with a \emph{parabolic structure} along $D$. The latter is described by a filtration	\[E={{E}_{{{x}_{i}},1}}\supseteq {{E}_{{{x}_{i}},2}}\supseteq {{E}_{{{x}_{i}},3}}=\left\{ 0 \right\}\]
together with a collection of real numbers called \emph{weights},  $0\le {{\alpha }_{1}}\left( {{x}_{i}} \right)\le {{\alpha }_{2}}\left( {{x}_{i}} \right)<1$, over each point ${{x}_{i}}\in D$. The \emph{multiplicity} of a weight ${{\alpha }_{j}}\left( {{x}_{i}} \right)$ is defined as the number ${{m}_{j}}\left( {{x}_{i}} \right)=\dim\left( {{{E}_{{{x}_{i}},j}}}/{{{E}_{{{x}_{i}},j+1}}}\; \right)$, for $j=1,2$ and $1\le i\le s$.
\end{defn}

We shall denote by $\left( E,\alpha  \right)$ a parabolic vector bundle of rank 2 over $\left( X,D \right)$ equipped with a parabolic structure determined by a system of weights $\alpha =\left( \alpha_{1}\left( x_i \right),\alpha_{2}\left( x_i \right) \right)$ at each ${{x}_{i}}\in D$. Natural parabolic structures can be constructed to obtain a notion of dual, subbundle, direct sum, tensor product or exterior product of parabolic bundles; see for instance \S 2.1 in \cite{KSZ2} for a detailed description of these constructions.  Note that these notions were defined in \cite{Sesh} and \cite{Yoko2} for any parabolic structure. The notation $E^{\vee}$ shall be used throughout the text to denote the parabolic dual of a parabolic bundle $E$.

\begin{defn}\label{202}
Let $\left( E, \alpha \right)$ and $\left( F,\beta  \right)$ be parabolic vector bundles over the pair $\left( X,D \right)$.  A holomorphic map $f:E\to F$ is called \emph{parabolic} if ${{\alpha }_{i}}\left( x \right)>{{\beta}_{j}}\left( x \right)$ implies $f\left( {{E}_{x,i}} \right)\subset {{F}_{x,j+1}}$ for every $x\in D$. Furthermore, we call such map \emph{strongly parabolic} if ${{\alpha }_{i}}\left( x \right)\ge {{\beta}_{j}}\left( x \right)$ implies $f\left( {{E}_{x,i}} \right)\subset {{F}_{x,j+1}}$ for every $x\in D$; this means alternatively that $f$ is a meromorphic map with at most simple poles along the divisor $D$ and with  $\text{Re}{{\text{s}}_{x}}f $ nilpotent for each $x\in D$.
\end{defn}

A stability condition for parabolic vector bundles is defined in \cite{Simp} with respect to the following notions for \emph{parabolic degree} and \emph{parabolic slope} of a parabolic vector bundle $\left( E,\alpha  \right)$ over $\left( X,D \right)$:

\[par\deg \left( E \right)=\deg E+{\sum\limits_{i=1}^{s}{\left[ {{\alpha }_{1}}\left( {{x}_{i}} \right)+{{\alpha }_{2}}\left( {{x}_{i}} \right) \right]}}\]
\[par\mu \left( \text{E} \right)=\frac{par\text{deg}\left( E \right)}{\text{rk}\left( E \right)}\]

\begin{defn}
A parabolic vector bundle $E$ over the pair $(X, D)$ is called \emph{stable} (resp. \emph{semistable}), if for every non-trivial parabolic subbundle $V\subseteq E$, it is $par\mu \left( V \right)<par\mu \left( E \right)$ (resp.$\le $).
\end{defn}

\subsection{Parabolic rank 2 $G_c$-Higgs bundles for $G_c$ complex.}
We now introduce moduli spaces of the families of pairs we are studying in this article.

\begin{defn}[parabolic $\text{GL}(2,\mathbb{C})$-Higgs bundle]\label{203}
A \emph{(strongly) parabolic $\text{GL}\left( 2,\mathbb{C} \right)$-Higgs bundle} over $\left( X,D \right)$ is defined as a pair $\left( E,\Phi  \right)$ consisting of a parabolic rank 2 vector bundle $E$ over $X$ and a (strongly) parabolic homomorphism $\Phi :E\to E\otimes {{K}}\left( D \right)$ called the \emph{Higgs field}.
\end{defn}

The stability condition for pairs  $(E,\Phi)$ is defined as follows:

\begin{defn}\label{205}
A (strongly) parabolic $\text{GL}\left( 2,\mathbb{C} \right)$-Higgs bundle $\left( E,\Phi  \right)$ will be called \emph{stable} (resp. \emph{semistable}), if for every non-trivial $\Phi $-invariant parabolic subbundle $V\subseteq E$, it holds that  $par\mu \left( V \right)<par\mu \left( E \right)$ (resp.$\le $); by $\Phi $-invariant it is meant here that $\Phi \left( V \right)\subseteq V\otimes K\left( D \right)$. The pair $(E,\Phi)$ is \emph{polystable}, if it is a direct sum of lower rank parabolic Higgs bundles of the same slope.
\end{defn}

Fix the topological invariants $d=par \deg E, n=\text{rk} E$ and a positive integer $m$. The \emph{moduli space} $\mathsf{\mathcal{M}}_{par}(d,n,m)=\mathsf{\mathcal{M}}\left( d;\alpha  \right)$ is defined as the set of isomorphism classes of polystable parabolic $\text{GL}\left( 2,\mathbb{C} \right)$-Higgs bundles over $(X,D)$ such that $d=par \deg E, n=\text{rk} E$ and all weights can be written as a fraction with denominator $m$.  Note that this fraction is not necessarily a reduced fraction. This moduli space was first constructed using Geometric Invariant Theory methods by K. Yokogawa in \cite{Yoko1}, who next showed that it is a complex quasi-projective variety, smooth at the stable points (see \cite{Yoko2}).

\begin{rem}\label{206}
We say that the weights are \emph{generic}, when every semistable parabolic Higgs bundle is automatically stable. We also call the parabolic structure of $E$ over $x$ \emph{trivial} if $\alpha_1(x)=\alpha_2(x)$, in other words, the filtration is trivial. The parabolic structure of $E$ over $x$ is called \emph{full} if $\alpha_1(x) \neq \alpha_2(x)$, that means, the filtration is full.
\end{rem}

Additional assumptions are now needed for the notion of parabolic $\text{SL}\left( 2,\mathbb{C} \right)$ and parabolic $\text{PGL}\left( 2,\mathbb{C} \right)$-Higgs bundles. We follow P. Gothen and A. Oliveira from \cite{GoOl}:

\begin{defn}[parabolic $\text{SL}(2,\mathbb{C})$-Higgs bundle]\label{207}
Let $\Lambda \in \text{Pi}{{\text{c}}^{d}}\left( X \right)$ be a holomorphic line bundle over $X$ with $\deg \Lambda =d$. A \emph{(strongly) parabolic $\text{SL}\left( 2,\mathbb{C} \right)$-Higgs bundle} over $\left( X,D \right)$ is defined as a pair $\left( E,\Phi  \right)$ consisting of a parabolic rank 2 vector bundle $E$ with ${{\wedge }^{2}}E\cong \Lambda $ and a (strongly) parabolic Higgs field $\Phi :E\to E\otimes K\left( D \right)$, such that $\text{Tr}\left( \Phi  \right)=0$.
\end{defn}

\begin{rem}\label{208}
The line bundle $\Lambda$ here is not considered to be necessarily trivial. Moreover, for the exterior product ${{\wedge }^{2}}E$ we suppress its natural parabolic structure inherited from the one on $E$.
\end{rem}

\begin{defn}[parabolic $\text{PGL}(2,\mathbb{C})$-Higgs bundle]\label{211}
A \emph{(strongly) parabolic $\text{PGL}\left( 2,\mathbb{C} \right)$-Higgs bundle} over $\left( X,D \right)$ is defined by an equivalence class $\left[ \left( E,\Phi  \right) \right]$ of parabolic $\text{GL}\left( 2,\mathbb{C} \right)$-Higgs bundles, where two pairs $\left( E,\Phi  \right)$ and $\left( {E}',{\Phi }' \right)$ are considered equivalent if there exists a parabolic line bundle $B$ over $X$, such that ${E}'\cong E\otimes B$ with the induced parabolic structure and ${\Phi }'=\Phi \otimes {{1}_{B}}$.
\end{defn}

\begin{rem}\label{212}
In the non-parabolic case, for the fixed degree $d$ holomorphic line bundle $\Lambda \to X$, let $d\equiv c\bmod 2$. Then, any holomorphic $\text{PGL}\left( 2,\mathbb{C} \right)$-bundle $E\to X$ with $\deg E=c$ lifts to a holomorphic vector bundle  $V\to X$ with ${{\wedge }^{2}}V\cong \Lambda $. Moreover, for any two such lifts  $V$ and  ${V}'$, there is a line bundle $L\to X$ with ${{L}^{2}}\cong {{\mathsf{\mathcal{O}}}_{X}}$, such that $V\cong {V}'\otimes L$.
\end{rem}

\begin{defn}\label{209}
Let $G_c$ be one of the two complex Lie groups $\text{GL}\left( 2,\mathbb{C} \right)$ or $\text{SL}\left( 2,\mathbb{C} \right)$. Fix a parabolic degree $d$ and a positive integer $m$. Then the \textit{moduli space of (strongly) parabolic $G_c$-Higgs bundles} is defined as the collection of isomorphism classes of semistable (strongly) parabolic $G_c$-Higgs bundles. We shall denote the moduli of polystable parabolic $G_c$-Higgs bundles by $\mathcal{M}_{par}(G_c,d,m)$ and that of strongly parabolic by $\mathcal{M}_{par}^{s}(G_c,d,m)$.
\end{defn}

The moduli space $\mathcal{M}_{par}(G_c,d,m)$ parameterizes semistable parabolic $G_c$-Higgs bundles, of which the weights can be written as a fraction with denominator $m$. Note that we are not fixing a particular parabolic structure. In fact, we have
\begin{align*}
\mathcal{M}_{par}(G_c,d,m)=\bigcup\limits_{\alpha}\mathcal{M}_{par}(G_c,d,\alpha),
\end{align*}
where $\alpha$ is a parabolic structure such that the weights are a fraction with denominator $m$, and $\mathcal{M}_{par}(G_c,d,\alpha)$ is the moduli space constructed by K. Yokogawa in \cite{Yoko1} parameterizing semistable parabolic $G_c$-Higgs bundles with a fixed parabolic structure $\alpha$, and the union is taking all such parabolic structures $\alpha$.

\begin{notn}\label{212.1}
In this paper we are interested in the case when $d=0$ and $m=2$. To simplify notation, define $\mathcal{M}_{par}(G_c):=\mathcal{M}_{par}(G_c,0,2)$ to be the moduli space.
\end{notn}

The moduli space of strongly parabolic $\text{SL}\left( 2,\mathbb{C} \right)$-Higgs bundles over $\left( X,D \right)$ can be also viewed as a fiber of the determinant map of the moduli space of strongly parabolic $\text{GL}\left( 2,\mathbb{C} \right)$-Higgs bundles. Indeed, for
\begin{align*}
p:\mathsf{\mathcal{M}}_{par} \left( \text{GL}\left( 2,\mathbb{C} \right) \right) & \to \text{Pi}{{\text{c}}^{d}}\left( X \right)\times {{H}^{0}}\left( X,K(D) \right)\\
\left( E,\Phi  \right) & \mapsto \left( {{\wedge }^{2}}E,\text{Tr}\left( \Phi  \right) \right)
\end{align*}
it is then $\mathsf{\mathcal{M}}_{par}\left( \text{SL}\left( 2,\mathbb{C} \right) \right)={{p}^{-1}}\left( \Lambda ,0 \right)$.

Note that the 2-torsion points of the Jacobian $\text{Jac}\left( X \right)$ are equivalent to spin structures on the surface $X$:
	\[{{\Gamma }_{2}}\cong {{H}^{1}}\left( X,{{\mathbb{Z}}_{2}} \right)\cong \mathbb{Z}_{2}^{2g}.\]
In order to obtain the moduli space of parabolic $\text{PGL}\left( 2,\mathbb{C} \right)$-Higgs bundles over $\left( X,D \right)$, we consider the action of the group ${{\Gamma }_{2}}$ on the fiber ${{p}^{-1}}\left( \Lambda ,0 \right)$ considered above given by
	\[\gamma \cdot \left( E,\Phi  \right)=\left( E\otimes {{L}_{\gamma }},\Phi \otimes {{1}_{{{L}_{\gamma }}}} \right),\]
for an element $\gamma \in {{\Gamma }_{2}}$ viewed as an abstract group, and a line bundle ${{L}_{\gamma }}\in {{\Gamma }_{2}}$ viewed as a 2-torsion point of $\text{Jac}\left( X \right)$. This action is trivial on the parabolic structure of $E$, since ${{L}_{\gamma }}$ is not considered to be equipped with any parabolic structure. We now define:

\begin{defn}\label{213}
Fix a holomorphic line bundle $\Lambda $ on $X$ of degree $d\in \mathbb{Z}$. The \emph{moduli space} of polystable (strongly) parabolic  $\text{PGL}\left( 2,\mathbb{C} \right)$-Higgs bundles with fixed degree $c\equiv d\bmod 2$, of which the weights can be written as a fraction with denominator $2$, is defined as the quotient space
	\[\mathsf{\mathcal{M}}_{par,c}^{(s)}\left( \text{PGL}\left( 2,\mathbb{C} \right) \right)={{{p}^{-1}}\left( \Lambda ,0 \right)}/{{{\Gamma }_{2}}}\;.\]
\end{defn}
This space is not smooth, but is an orbifold with singularities arising from the fixed points of the action of the group ${{\Gamma }_{2}}$.

\subsection{Parabolic rank 2 $G$-Higgs bundles for $G$ real.}
We next introduce parabolic $G$-Higgs bundles, for the groups $G=\text{GL}(2,\mathbb{R})$, $\text{SL}(2,\mathbb{R})$ and $\text{PGL}(2,\mathbb{R})$.

\begin{defn}\label{214}
Let $(X,D)$ be a pair of a Riemann surface and a divisor as considered earlier.
\begin{enumerate}
\item A parabolic $\text{GL}(2,\mathbb{R})$-Higgs bundle over $(X,D)$ is a pair $(E,\Phi)$, where $E$ is a rank $2$ parabolic bundle over $X$ equipped with an orthogonal structure $$\langle ,\rangle : E \otimes E \rightarrow \mathbb{C},$$
where the tensor product is the parabolic tensor product, and $\Phi$ is a parabolic Higgs field which is symmetric with respect to the orthogonal structure $\langle,\rangle$.
\item A parabolic $\text{SL}(2,\mathbb{R})$-Higgs bundle over $(X,D)$ is a triple $(N,\beta,\gamma)$, where $N$ is a parabolic line bundle and $\beta \in H^0(X,N^2 \otimes K(D))$, $\gamma \in H^0(X,N^{-2}\otimes K(D))$.
\item A parabolic $\text{PGL}(2,\mathbb{R})$-Higgs bundle over $(X,D)$ is an equivalence class of a triple $(E,\Phi,A)$, where $E$ is a rank $2$ parabolic bundle with a symmetric, non-degenerate bilinear pairing $$\langle ,\rangle : E \otimes E \rightarrow A$$ with values in a parabolic line bundle $A$ and $\Phi$ is a parabolic Higgs field. We say $(E,\Phi,A)$ and $(E',\Phi',A')$ are \emph{equivalent} if there is a parabolic line bundle $B$ such that $(E,\Phi,A)=(E'\otimes B, \Phi' \otimes \text{Id}_B, A \otimes B^2)$.
\end{enumerate}
\end{defn}

\begin{defn}\label{215}
The stability condition for parabolic $G$-Higgs bundles is induced from the stability condition in Definition \ref{205}. Accordingly, we define:
\begin{enumerate}
\item A parabolic $\text{GL}(2,\mathbb{R})$-Higgs bundle $(E,\Phi)$ is called \emph{stable} (resp. \emph{semistable}) if for any $\Phi$-invariant parabolic line bundle $V \subseteq E$, it is
    \begin{align*}
    par \deg(V)<0 \quad (\text{resp. } par \deg(V) \leq 0).
    \end{align*}
    We say $(E,\Phi)$ is \emph{polystable} if either $(E,\Phi)$ is stable, or $E=N \bigoplus N^{\vee}$ and $\Phi=\phi. \text{Id}$ for some $\phi \in H^0(X,K(D))$, where the orthogonal structure comes from the dual pairing on $N$ and $N^{\vee}$.
\item A parabolic $\text{SL}(2,\mathbb{R})$-Higgs bundle $(N,\beta,\gamma)$ is called \emph{stable} (resp. \emph{semistable, polystable}), if the corresponding $\text{GL}(2,\mathbb{R})$-Higgs bundle is stable (resp. semistable, polystable).
\item A parabolic $\text{PGL}(2,\mathbb{R})$-Higgs bundle $(E, \Phi, A)$ is called \emph{stable} (resp. \emph{semistable}), if for any $\Phi$-invariant parabolic line bundle $N \subseteq E$, it is
    \begin{align*}
    par \deg(N)< \frac{par \deg(A)}{2} \quad (\text{resp. } par \deg(N) \leq \frac{par \deg(A)}{2}).
    \end{align*}
    A parabolic $\text{PGL}(2,\mathbb{R})$-Higgs bundle $(E,\Phi,A)$ is called \emph{polystable}, if it is stable, or $\Phi=0$ and $E=N_1 \bigoplus N_2$, where
    \begin{align*}
    par \deg(N_1)=par \deg(N_2)=\frac{par \deg(A)}{2},\quad A=L_1 L_2,
    \end{align*}
    and the orthogonal structure on $E$ comes from the pairing on $N_1$ and $N_2$.
\end{enumerate}
\end{defn}

\begin{rem}\label{250}
In fact, one can consider more generally $L$-twisted (strongly) parabolic $G$-Higgs bundles $(E, \Phi)$ for (strongly) parabolic Higgs fields $\Phi: E \to E \otimes{L} $, where $L$ is a parabolic line bundle with $l = par \deg (L) \ge 2g-2+s$. In this article we shall discuss only the case when $L\cong K(D)$, that means, $l=2g-2+s$. However, the main results in \S 4, 5 and 7 later on can be obtained for moduli spaces of $L$-twisted parabolic $G$-Higgs bundles using exactly the same approach.
\end{rem}

\begin{defn}\label{216}
Let $G$ be one of the real Lie groups $\text{GL}(2,\mathbb{R})$, $\text{SL}(2,\mathbb{R})$ or $\text{PGL}(2,\mathbb{R})$. We shall denote by $\mathcal{M}_{par}(G)$ the moduli space of isomorphism classes of polystable parabolic $G$-Higgs bundles over the pair $(X,D)$ such that the weights in the parabolic structure can be written as a fraction (not necessarily reduced) with denominator $2$.
\end{defn}

\begin{rem}\label{210}
The following properties about complex and real Higgs bundles can be checked.
\begin{enumerate}
\item For the rank 2 real Lie groups $G=\text{GL}(2,\mathbb{R})$, $\text{SL}(2,\mathbb{R})$ and $\text{PGL}(2,\mathbb{R})$,  a parabolic $G$-Higgs bundle is naturally a parabolic $G_c$-Higgs bundle, for $G_c = \text{GL}(2,\mathbb{C}), \text{SL}(2,\mathbb{C})$ or $\text{PGL}(2,\mathbb{C})$ respectively.
\item A parabolic $\text{SL}(2,\mathbb{R})$-Higgs bundle $(N,\beta,\gamma)$ can be considered as a $\text{GL}(2,\mathbb{R})$-Higgs bundle $(E,\Phi)$, where $E=N\bigoplus N^{\vee}$ and $\Phi=\begin{pmatrix} 0 & \beta \\ \gamma & 0 \end{pmatrix}$ is a trace free parabolic Higgs field. In fact, if a $\text{GL}(2,\mathbb{R})$-Higgs bundle $(E,\Phi)$ has the above property, then it has a natural $\text{SL}(2,\mathbb{R})$-structure. Therefore, we may use the data $(N,\beta,\gamma)$ to represent this $\text{SL}(2,\mathbb{R})$-Higgs bundle.
\item The moduli space of parabolic $\text{SL}(2,\mathbb{R})$-Higgs bundles is alternatively obtained considering fixed points of a holomorphic involution: A parabolic $\text{SL}(2,\mathbb{R})$-Higgs bundle is an $\text{SL}(2,\mathbb{C})$-Higgs bundle that is fixed under the transpose map
\begin{align*}
    \Theta_{\text{SL}(2,\mathbb{R})}:  \mathcal{M}_{par}(\text{SL}(2,\mathbb{C})) & \rightarrow \mathcal{M}_{par}(\text{SL}(2,\mathbb{C}))\\
     (E,\Phi) & \rightarrow  (E^*,\Phi^t).
    \end{align*}
\item In the non-parabolic case, let $(E,\Phi,A)$ be a $\text{PGL}\left( 2,\mathbb{R} \right)$-Higgs bundle and let ${{\wedge }^{2}}E\cong \Lambda $ for $\deg E=c$ fixed. The $\text{PGL}\left( 2,\mathbb{R} \right)$-Higgs bundle $(E,\Phi,A)$ is equivalent to a $\text{PGL}\left( 2,\mathbb{R} \right)$-Higgs bundle $(E',\Phi',A')$, with $\deg(E')=0$ or $1$ (see \cite{BarScha}). In the parabolic case any parabolic $\text{PGL}\left( 2,\mathbb{R} \right)$-Higgs bundle $(E,\Phi,A)$ is equivalent to a parabolic $\text{PGL}\left( 2,\mathbb{R} \right)$-Higgs bundle $(E',\Phi',A')$, where $par \deg E=0$ or $\frac{1}{2}$.
\item Let $(E,\Phi,A)$ be a stable $\text{PGL}(2,\mathbb{R})$-Higgs bundle. Suppose that $E=N_1 \oplus N_2$ can be decomposed into two parabolic line bundles. In this case $\Phi$ is a strictly upper (or lower) triangular matrix with respect to the decomposition $E=N_1 \oplus N_2$. It is easy to check that $N_1 \otimes A^{\vee}$ and $N_2 \otimes A^{\vee}$ are parabolic dual to each other.

\end{enumerate}
\end{rem}

\section{Parabolic BNR Correspondence and Higgs $V$-Bundles}

In this section, we discuss the parabolic BNR correspondence. In \S 3.1 and \S 3.2, we discuss the parabolic BNR correspondence and the correspondence between parabolic bundles and $V$-bundles for general $m$, where $m$ is the denominator of the weights in the parabolic structures. In \S 3.3 and \S 3.4, we work on the case when $m=2$.

\subsection{The Parabolic Hitchin fibration}
We next consider the Hitchin fibration and spectral curve construction in the case of parabolic $G_c$-Higgs bundles $\left( E,\Phi  \right)$ over $\left( X,D \right)$, where $G_c$ is either of $\text{SL}(2,\mathbb{C})$, $\text{GL}(2,\mathbb{C})$ or $\text{PGL}(2,\mathbb{C})$. The Hitchin map for
parabolic Higgs bundles was first defined in \cite{LoMa} for non-strongly parabolic Higgs fields $\Phi$, while it was studied for strongly parabolic $\Phi$ in \cite{GoLog}. We review these constructions for the particular cases we are interested in next. A concrete description of the generic fibers of the parabolic Hitchin map can be also found in \cite{SWW}.

Let $\tot(K(D))$ be the total space of the line bundle $K(D)$ and let $\pi:\tot(K(D))\to X$ be the canonical projection. Let, lastly, $\lambda\in H^0(\pi^* K(D))$ be the tautological section. On the moduli space of parabolic $G_c$-Higgs bundles $\mathcal{M}_{par}(G_c)$ we consider the \emph{parabolic Hitchin fibration}
\begin{align*}
h:\mathcal{M}_{par}(G_c) \to \mathcal{H}=\bigoplus_{i=1}^2 H^0(X,(K(D))^i),
\end{align*}
which sends a pair $\left( E,\Phi  \right)$ to the pair of sections $\eta =\left( {{\eta }_{1}},{{\eta }_{2}} \right)\in \mathsf{\mathcal{H}}$, where the ${{\eta }_{i}}$ are defined as the coefficients of the characteristic polynomial of the Higgs field $\Phi$
\[\det(\lambda\cdot \mathrm{id}-\pi^* \Phi)=\lambda^2+\eta_1 \lambda +\eta_2 \in H^0(\tot(K(D)),\pi^*(K(D)^2)).\]
Note that for $\Phi $ strongly parabolic, its eigenvalues vanish at $D$. As an example, on the moduli space of parabolic  $\SL(2,\mathbb{C})$-Higgs bundles, the parabolic Hitchin fibration is
$h:\mathcal{M}_{par}(\text{SL}(2,\mathbb{C})) \to H^0(X,(K(D))^2)$, sending
a pair $\left( E,\Phi  \right)$ to ${\eta }_{2} \in H^0(X,(K(D))^2)$. In this case the trace of $\Phi$ vanishes, thus the characteristic polynomial is $\lambda^2+\eta_2$, i.e. $\eta_1=0$.

Given a point $\eta=(\eta_1, \eta_2)\in \mathcal{H}$, the \textit{spectral curve} $X_\eta$ is defined as the zero set of the characteristic polynomial $\det(\lambda\cdot \mathrm{id} -\pi^*\Phi)=0$. The restriction of $\pi$ to $X_\eta$ gives a 2-cover of the Riemann surface $X$ and is ramified when the determinant has multiple roots. In the case of strongly parabolic Higgs bundles the cover is completely ramified over the parabolic points, since on these points the spectral curve is defined by $\lambda^2=0$.

By a Bertini theorem argument (also Lemma 3.1 in \cite{GoLog}), there is an open and dense subset $H_{reg}\subset \mathcal{H}$, such that $X_\eta$ is smooth when $\eta\subset H_{reg}$. In this paper we slightly change the definition of the regular locus of the classical non-parabolic Hitchin fibration and denote by $H_{reg}$ the set of elements $\eta \in \mathcal{H}$ such that the corresponding spectral curve $X_\eta$ is smooth and the intersection between the set of branch points $B$ (of the spectral covering $X_\eta \rightarrow X$) and the set $D$ is empty. We call $H_{reg}$ the \emph{regular locus of the parabolic Hitchin fibration}. If an element $\eta$ is in $H_{reg}$, we call this element $\eta$ a \emph{generic element}. Clearly, $H_{reg}$ is a subset of the regular locus of the non-parabolic Hitchin fibration and it is still an open dense subset in $\mathcal{H}$; we are actually adding the condition $B \bigcap D = \emptyset$ in the definition of $H_{reg}$ to exclude the strongly parabolic case.

For a generic element in $\mathcal{H}$, the genus ${{g}_{{{X}_{\eta }}}}$ of the spectral curve $X_\eta$  can be calculated by the adjunction formula as follows
\begin{align*}
2{{g}_{{{X}_{\eta }}}}-2  =& \deg(K_{X_\eta})=K_{X_\eta}\cdot X_\eta \\
=&(K_{\text{tot}(K(D))}+X_\eta) \cdot X_\eta\\
=& 2 c_1(\mathcal{O}(-D))+2^2 X_\eta^2\\
=& -2s+4(2g-2+s)\\
=& 8g-8+2s,
\end{align*}
where $g$ denotes the genus of $X$ and $s$ the number of points in $D$. This provides the genus of the spectral curve
\begin{align*}
{{g}_{{{X}_{\eta }}}}=4g-3+s.
\end{align*}

By Lemma 3.2 of \cite{GoLog}, if we fix a parabolic structure $\alpha$, then the fiber of the parabolic Hitchin fibration over a generic element $\eta$ is a torsor over the Prym variety
$$\mathrm{Prym}(X_\eta,X):=\{L\in \mathrm{Pic}(X_\eta): \det \pi_* L=\mathcal{O}_X\}.$$
Indeed,
$$h^{-1}(\eta)=\{L\in \mathrm{Pic}(X_\eta): \det \pi_* L=\det E\}.$$

Let $\alpha(x)$ be the corresponding parabolic structure for each point $x \in D$. A parabolic line bundle $(L,\widetilde{\alpha})$ over $X_\eta$ is called \emph{compatible with the given parabolic structure} $\alpha$, if
\begin{align*}
\{\widetilde{\alpha}(\widetilde{x}), \widetilde{x} \in \pi^{-1}(x)\}=\{\alpha_i(x), 1 \leq i \leq 2\}.
\end{align*}

The following Proposition provides the Beauville-Narasimhan-Ramanan correspondence in the parabolic case:
\begin{prop}[Proposition 5.2 in \cite{KSZ2}]\label{301}
Let $X$ be a smooth Riemann surface of genus $g$ and let $D$ be a reduced effective divisor on $X$. Fix a parabolic structure $\alpha$ for a rank 2 parabolic bundle over $X$ and a pair of sections $\eta=\left( {{\eta }_{1}},{{\eta }_{2}} \right)$, where $\eta_i$ is a section of $K(D)^i$ for $1 \leq i \leq 2$. Suppose that the surface $X_\eta$ is non-singular and the intersection of the branch points $B$ and the given divisor $D$ is empty. If we fix an order for the pre-image $\{\widetilde{x}_1, \widetilde{x}_2\}$ of each $x \in D$, then there is a bijective correspondence between isomorphism classes of parabolic line bundles $(L,\widetilde{\alpha})$ on $X_\eta$ such that
\begin{align*}
\widetilde{\alpha}(\widetilde{x}_i)=\alpha_i(x),
\end{align*}
and isomorphism classes of pairs $(E, \Phi)$, where $E$ is a parabolic bundle of rank 2 with parabolic structure $\alpha$ and $\Phi:E \rightarrow E \otimes K(D)$ a parabolic Higgs field with characteristic coefficients $\eta_i$.
\end{prop}

This proposition says that given a rank $2$ bundle on $X$, we may get a parabolic line bundle on $X_\eta$; we shall briefly explain next how to construct a parabolic bundle on $X$ given a parabolic line bundle on $X_\eta$. Let $L$ be a parabolic line bundle on $X_\eta$. If we forget the parabolic structure, we get a rank $2$ bundle $E$ on $X$ from the classical BNR correspondence. Now, for a point $x \in D$ with pre-images $\{\widetilde{x}_1, \widetilde{x}_2\}$, if $\widetilde{\alpha}(\widetilde{x}_1)=\widetilde{\alpha}(\widetilde{x}_2)$, then we define the parabolic structure on $E|_x$ as the trivial filtration with weight $\widetilde{\alpha}(\widetilde{x}_2)$. If $\widetilde{\alpha}(\widetilde{x}_1) < \widetilde{\alpha}(\widetilde{x}_2)$, we define the parabolic structure on $E|_x$ as
\begin{align*}
 E  = & {{E}_{{{x}_{i}},1}}  \supseteq {{E}_{{{x}_{i}},2}}\supseteq {{E}_{{{x}_{i}},3}}=\left\{ 0 \right\}\\
 & 0  \le \widetilde{\alpha}(\widetilde{x}_1) < \widetilde{\alpha}(\widetilde{x}_2) <1.
\end{align*}

In the proposition above we fix the order of the sheets of the covering $X_\eta \rightarrow X$ and give a spectral parabolic structure $\widetilde{\alpha}$ for the corresponding line bundle $L$. This thus describes a one-to-one correspondence. In general the parabolic BNR correspondence may not be a one-to-one correspondence. The correspondence depends on the parabolic structure of the holomorphic bundle $E$.  We next provide two examples to make this situation more clear:
\begin{exmp}\label{302}
Let $(E,\Phi)$ be a parabolic $\text{GL}(2,\mathbb{C})$-Higgs bundle over $X$. Denote by $\alpha$ the parabolic structure and let $X_\eta$ be its spectral covering. Under the classical BNR correspondence, there is a line bundle $L$ over $X_\eta$, which is uniquely determined by $E$. We fix a point $x \in D$, and let $x',x''$ be its pre-images in $X_\eta$. We only discuss the parabolic structures of $E$ and $L$ over $x$ and $x',x''$ respectively.
\begin{enumerate}
\item Let $\alpha_1(x)=\alpha_2(x)$. The parabolic structures over $x'$ and $x''$ are both trivial with weight $\alpha_1(x)$. In this case, the parabolic structure of $L$ over $x',x''$ is uniquely determined. More generally, if the filtration is trivial, then the correspondence is one-to-one.
\item Let $\alpha_1(x)<\alpha_2(x)$. Then there are two possibilities for the parabolic structures over $x'$ and $x''$. The first is the weight of $L|_{x'}$ to be $\alpha_1(x)$ and the weight of $L|_{x''}$ to be $\alpha_2(x)$. The second is the weight of $L|_{x'}$ to be $\alpha_2(x)$ and the weight of $L|_{x''}$ to be $\alpha_1(x)$. In general if the filtration is not trivial, then the correspondence is not one-to-one.
\end{enumerate}
\end{exmp}
As discussed above, the parabolic BNR correspondence may not be a one-to-one correspondence and the correspondence is closely related to the parabolic structure. Thus we have to be more careful when describing the parabolic $\text{GL}(2,\mathbb{C})$-Hitchin fibration as a disjoint copy of the Picard group of the spectral covering $X_\eta$. More generally, the parabolic BNR correspondence may not be one-to-one for parabolic $G$-Higgs bundles, where $G$ is a real Lie group. In this paper we are not dealing with this problem for arbitrary $G$, but we rather consider the rank 2 cases $\text{SL}(2,\mathbb{R})$, $\text{GL}(2,\mathbb{R})$ and $\text{PGL}(2,\mathbb{R})$ for which it turns out that the BNR correspondence is in fact one-to-one when restricted to a particular subvariety (see the proofs of the forthcoming Propositions \ref{601}, \ref{701} and \ref{703} respectively); we are using the $V$-surface and $V$-Picard group introduced in the next subsection to exhibit this construction in further detail.

\subsection{Higgs $V$-Bundles and Spectral Covering of a $V$-Surface}
Fix a positive integer $m$. Let $M$ be the $V$-surface with underlying surface $X$ for a collection of marked points $D=\{x_1,\dots,x_s\}$. The local chart around $x_j$ is isomorphic to $\mathbb{D}^2/ \mathbb{Z}_{m}$, where $\mathbb{D}^2$ is the unit disk. The $V$-surface $M$ is an orbifold. With respect to this construction (see  \cite{FuSt} for details) the data $(X,D,m)$ can uniquely determine the $V$-surface $M$. We also assume that $M$ can be written as a global quotient. 

Given a $V$-surface $M$ with respect to the data $(X,D,m)$, we define the $V$-cohomology $H^{*}_V(M)$ in the following way. Define $M_V$ as a union of $M \backslash \{x_1,...,x_s\}$ and $\coprod  D_i \times_{\mathbb{Z}_{m}}E \mathbb{Z}_{m}$,  where $E \mathbb{Z}_{m}$ denotes the universal bundle over $B \mathbb{Z}_{m}$, the classifying space of $\mathbb{Z}_{m}$, and $D_i$ is a neighborhood around $x_i$ with $\mathbb{Z}_m$-action. These charts can be glued naturally. Note also that by $V$-surfaces we mean the orbifolds, and the $V$-cohomology is precisely orbifold cohomology. The $V$-cohomology $H^{*}_V(M,\mathbb{Z}_2)$ with coefficient $\mathbb{Z}_2$ is defined as
\begin{align*}
H^{*}_V(M,\mathbb{Z}_2)=H^{*}(M_V,\mathbb{Z}_2).
\end{align*}
Let $m=2$. The $\mathbb{Z}_2$-coefficient first $V$-cohomology is $H^{1}_V(M,\mathbb{Z}_2) \cong \mathbb{Z}_2^{2g+s-1}$ (see \S 7 in \cite{KSZ}). In \S 5 later on, we use the first $V$-cohomology $H^{1}_V(M,\mathbb{Z}_2)$ to define the lattice.

A parabolic bundle $E$ with rank $n$ over the Riemann surface $X$ can be pulled-back to an orbifold bundle $\mathcal{E}$ over $M$ using the flag data to construct local representations of cyclic groups on the fiber $\mathbb{C}^n$ over punctures in $D$; then $\mathcal{E}$ pushes forward to $E$. A parabolic bundle $E$ over $X$ which is a push forward has rational weights of the form $k/m$ in the flag over each parabolic point, where $m$ is the order of the cone point in $M$. H. Boden in \cite{Boden} called parabolic bundles with such rational weights \emph{commensurate with $M$} and showed that for a commensurate parabolic bundle $E$ over $X$ there exists a holomorphic orbifold bundle $\mathcal{E}$ over $M$, so that $E$ is the push forward of $\mathcal{E}$ (see Definition 4.3 and Proposition 4.4 in \cite{Boden}). There is a one-to-one correspondence between parabolic Higgs bundles over $(X,D)$ such that any weight can be written as a fraction with denominator $m$ and Higgs $V$-bundles over $M$, as pioneered in the work of B. Nasatyr and B. Steer \cite{NaSt}; see also \S 6 and 7 in \cite{KSZ} for a version of this correspondence in our setting.

Let $\eta$ be a section $H_{reg} \subseteq H^0(M,K(D)^2)$, where $K(D)$ here is actually the pull-back of the bundle $K(D)$ over $X$ to $M$. Let $(E_M,\Phi_M)$ be a rank $2$ Higgs $V$-bundle over $M$ such that the characteristic polynomial of $\Phi$ is $\lambda^2 + \eta$. Since we assumed that $M$ can be written as a global quotient $[Y/ \Gamma]$, it is equivalent to consider $\eta$ as a $\Gamma$-equivariant section in $H^0(X,K_Y^2)$ and $(E,\Phi)$ as a $\Gamma$-equivariant Higgs bundle over $Y$ (see \cite{Biswas2}).

Now we construct the $V$-surface $M_\eta$ as the zero set of $\lambda^2 + \eta$, which can be considered as the spectral covering of $M$. As a $\Gamma$-equivariant section, $\eta$ can be naturally considered as a section in $H^0(X,K(D)^2)$ under a change of coordinate $w=z^m$. Let $D_\eta$ be the pre-image $\pi^{-1}D$, where $\pi: X_\eta \rightarrow X$ is the natural projection. By the definition of $H_{reg}$, given $x \in D$, it has two pre-images $x',x'' \in D_\eta \subseteq X_\eta$. The local charts of $M_\eta$ are given as follows
\begin{align*}
& \phi_j: U_{x'_j} \rightarrow \mathbb{D}^2 / \mathbb{Z}_{m}, & 1 \leq j \leq s;\\
& \phi_j: U_{x''_j} \rightarrow \mathbb{D}^2 / \mathbb{Z}_{m}, & 1 \leq j \leq s;\\
& \phi_p: U_p \rightarrow \mathbb{D}^2, & p \in M \backslash \pi^{-1}(D),
\end{align*}
where $\mathbb{D}^2$ is the disk. In other words, the local charts around $x',x'' \in D_\eta$ are the same as that of $x=\pi(x')=\pi(x'') \in D$. This gives the $V$-surface $M_\eta$.

\begin{rem}
Let $(E,\Phi)$ be a strongly parabolic Higgs bundle and let $\lambda^2 + \eta$ be the characteristic polynomial for the Higgs field $\Phi$. The section $\eta \in H^{0}(X, K^2(D))$ is a section of $K(D)^2$ vanishing along $D$. Note that in the parabolic case, we always assume that the intersection of the zero set of the section $\eta$ and the divisor $D$ is empty. With respect to this assumption, we can construct the spectral covering $M_\eta$ of $M$ such that for $x\in M$, the pre-image $\pi^{-1}(x)$ always contains two points. However, in the strongly parabolic case the points in $D$ are also branch points, therefore the pre-image $\pi^{-1}(x)$ contains only one point.
\end{rem}

\subsection{Parabolic BNR Correspondence}
In subsection \S 3.1 we described a BNR correspondence for parabolic Higgs bundles. This correspondence also applies for Higgs $V$-bundles under the correspondence between parabolic Higgs bundles and Higgs $V$-bundles:
\begin{center}
\begin{tikzcd}	
	\text{Higgs $V$-bundles over $M$ } \arrow[r,"\text{ BNR }"] \arrow[d,"\text{ Parabolic Bundles vs. $V$-bundles }"] & \text{line $V$-bundles over $M_\eta$} \arrow[d,"\text{ Parabolic Bundles vs. $V$-bundles }"]\\

	 \text{ parabolic Higgs bundles} \arrow[r,"\text{ BNR }"] & \text{parabolic line bundles over $X_\eta$}
\end{tikzcd}	
\end{center}
Any parabolic Higgs bundle can be viewed as a Higgs $V$-bundle; thus in this article we shall be using the two terms interchangeably.

The $V$-Picard group $\text{Pic}_V(M_\eta)$ is defined as the group of isomorphism classes of line $V$-bundles; line $V$-bundles with faithful isotropy representations correspond to Seifert fibrations when the corresponding circle bundles are taken (see \cite{FuSt}). Constructing a one-to-one correspondence in the parabolic case is equivalent to finding a subvariety of the $V$-Picard group $\text{Pic}_V(M_\eta)$ such that there is a one-to-one correspondence between the points (line $V$-bundles) in the subvariety and the Higgs $V$-bundles over $M$.

When $m=2$, the local chart around each point $x \in D=\{x_1,\dots,x_s\} \subseteq M$ is isomorphic to $\mathbb{D}^2/\mathbb{Z}_2$. We have the following short exact sequence for the $V$-Picard group
\begin{align*}
    0 \rightarrow \text{Pic}(X) \rightarrow \text{Pic}_V(M) \rightarrow \bigoplus_{i=1}^{s} \mathbb{Z}_2 \rightarrow 0,
\end{align*}
where the last term $\bigoplus_{i=1}^{s} \mathbb{Z}_2$ is parameterizing all parabolic structures of line $V$-bundles (see \cite[\S 1]{FuSt}). Also, the subvariety of all line $V$-bundles with a given parabolic structure $\alpha$ is isomorphic to $\text{Pic}(X)$.

In the case when $m=2$ and in view of Example \ref{302}, for rank 2 parabolic bundles with weights in the set $\{0,\frac{1}{2}\}$, we find that the weights in the filtration of the pre-image of each point $x\in D$ are either all the same (thus all equal to 0 or all equal to $\frac{1}{2}$), or they are not (thus one is equal to 0 and the other equal to $\frac{1}{2}$). We say that the filtration is of \emph{Type 1} if the weights in the filtration are the same, and of \emph{Type 2} otherwise.

Fix the weight type $\beta(x)$ for each point $x \in D$ and assume that there are $j$-many points in the divisor $D$ for which the weights are all the same. For these points there are $2^{2j}$ many choices for the parabolic structure, while there are $2^{s-j}$ many choices for the points in $D$ for which the weights are not the same. Let $j(\beta)=s+j$. Thus the subvariety of line bundles in this case is isomorphic to $\prod\limits_{i=1}^{j(\beta)} \text{Pic}(X)$.  Taking the union over all possible types, we construct the subvariety
\begin{align*}
\bigcup_{\beta}\prod\limits_{i=1}^{j(\beta)} \text{Pic}(X)
\end{align*}
and the parabolic BNR correspondence is a one-to-one correspondence when restricted to this subvariety of $\text{Pic}_V(M_\eta)$.

We complete this subsection reviewing some properties of the parabolic BNR correspondence, which will be used to study the $2$-torsion points in certain subvarieties of $\text{Pic}_V(M_\eta)$. Let $(E,\Phi)$ be a parabolic Higgs bundle over $(X,D)$ and let $X_\eta$ be the corresponding spectral curve. If we forget the parabolic structure, the classical BNR correspondence (Proposition 3.6 in \cite{BNR}) gives the existence of a unique line bundle $L$ over $X_\eta$ such that $\pi_*(L)=E$ and the Higgs field $\Phi$ is induced by the tautological section $\lambda: L \rightarrow L \otimes \pi^*(K(D))$. Instead of studying the line bundle $L$, we construct the line bundle $L_0$ as follows. We fix a square root $K(D)^{\frac{1}{2}}$ of $K(D)$ (as line $V$-bundle) and write $L=L_0 \otimes \pi^*K(D)^{\frac{1}{2}}$. Thus a parabolic Higgs bundle $(E,\Phi)$ corresponds to a line $V$-bundle $L_0$. The reason why we do this modification is that when the parabolic degree of $E$ is zero, then the parabolic degree of $L_0$ is also zero. Also, note that the orthogonal structure gives $(E, \Phi) \cong (E^{\vee}, \Phi^*)$, which implies that $L_0 \cong L_0^{\vee}$, that is, $L_0^2 =\mathcal{O}$ as parabolic line bundles.

\subsection{Prym Variety}
In order to study parabolic $\text{SL}(2,\mathbb{R})$-Higgs bundles of which the weights are with denominator $2$, we introduce the Prym variety as the set of line $V$-bundles over $M_\eta$ such that
\begin{align*}
\text{Prym}(M_\eta,M)=\{L \in \text{Pic}_V(M_\eta)| \tau^*L =L^{-1}\},
\end{align*}
which can be considered as the Prym variety of the spectral covering $M_\eta \rightarrow M$, where $M$ and $M_\eta$ are the corresponding $V$-surfaces of $X$ and $X_\eta$, and $\tau:M_\eta \rightarrow M_\eta$ is the involution of $M_\eta$.

Recall that we have the following short exact sequences for $V$-Picard groups $\text{Pic}_V(M)$ and $\text{Pic}_V(M_\eta)$
\begin{align*}
    0 \rightarrow \text{Pic}(X) \rightarrow \text{Pic}_V(M) \rightarrow \bigoplus_{i=1}^{s} \mathbb{Z}_2 \rightarrow 0,\\
    0 \rightarrow \text{Pic}(X_\eta) \rightarrow \text{Pic}_V(M_\eta) \rightarrow \bigoplus_{i=1}^{2s} \mathbb{Z}_2 \rightarrow 0.
\end{align*}
These exact sequences together with the condition of involution give us the following exact sequence for the Prym variety $\text{Prym}(M_\eta,M)$:
\begin{align*}
0 \rightarrow \text{Prym}(X_\eta,X) \rightarrow \text{Prym}(M_\eta,M) \rightarrow \bigoplus_{i=1}^{s} \mathbb{Z}_2 \rightarrow 0.
\end{align*}

\section{Topological invariants for moduli of rank 2 parabolic $G$-Higgs bundles}
Moduli spaces of parabolic (or non-parabolic) $G$-Higgs bundles can be decomposed into closed subvarieties, yet not necessarily connected components, for fixed values of appropriate topological invariants. In this section we describe such topological invariants for the moduli spaces we are interested in, namely for the rank 2 cases $\text{SL}\left( 2,\mathbb{R} \right)$, $\text{GL}\left( 2,\mathbb{R} \right)$ and $\text{PGL}\left( 2,\mathbb{R} \right)$. The computation of the total number of these invariants directly provides a count for the minimum number of connected components of these moduli spaces. Then, in the following sections we study the orbits of the monodromy action on the relative parabolic Hitchin systems to obtain an exact count of these components.

\subsection{Topological invariants in the case when $G=\text{SL}\left( 2,\mathbb{R} \right)$}
Let  $\left( E,\Phi  \right)$ be a parabolic $\text{SL}\left( 2,\mathbb{R} \right)$-Higgs bundle, that is, $E=N \oplus N^{\vee}$ and $\Phi =\left( \begin{matrix}
   0 & \beta   \\
   \gamma  & 0  \\
\end{matrix} \right)$, for a line bundle $N \to X$ equipped with a parabolic structure. The rational number $\tau = par\deg N$ defines a topological invariant, called the \emph{parabolic Toledo invariant}, first introduced in \cite{GaGoMu} and subsequently defined in \cite{BiGaRi} for $G$-Higgs bundles with $G/H$ a Hermitian symmetric space of noncompact type, where $H\subset G$ is a maximal compact subgroup. In \S 8.2 of \cite{BiGaRi} a general inequality of Milnor-Wood type was proven, while in Proposition 5.4 of \cite{KSZ} it is shown more explicitly that in the particular case of a semistable parabolic  $\text{SL}\left( 2,\mathbb{R} \right)$-Higgs bundle $\left( E,\Phi  \right)$, this rational number $\tau$ satisfies the inequality
\[\left| \tau  \right|\le g-1+\frac{s}{2}.\]
We prove that for each non-maximal value of this topological invariant, i.e. for $\left| \tau  \right| < g-1+\frac{s}{2}$, there is at least one connected component of the moduli space ${{\mathsf{\mathcal{M}}}_{par}}\left( \text{SL}\left( 2,\mathbb{R} \right) \right)$; this implies the existence of at least $$2\cdot \left( g-1+\frac{s}{2} \right)-1$$ topological invariants of a parabolic $\text{SL}\left( 2,\mathbb{R} \right)$-Higgs bundle and thus we have $2g-3+s$ many topological invariants for any fixed parabolic structure.

By definition, the parabolic degree is the sum of the degree of the bundle plus the contribution from the weights in the parabolic structure. In fact, we may interpret this definition from the exact sequence for $\text{Pic}_V(M)$ mentioned earlier:
\begin{align*}
 0 \rightarrow \text{Pic}(X) \rightarrow \text{Pic}_V(M) \rightarrow \bigoplus_{x \in D} \mathbb{Z}_2 \rightarrow 0.
\end{align*}
The degree $d \in \text{Pic}(X)$ and the weights $w \in \bigoplus\limits_{x\in D} \mathbb{Z}_2$ are topological invariants for the line $V$-bundle. Although different pairs $(d,w)$ may give us the same parabolic degree, they provide different topological invariants for line $V$-bundles. Note that there are $2^s$ many choices for the parabolic structure. Therefore the total number of topological invariants for the non-maximal values of $\tau$ is $2^s\cdot(2g-3+s)$.

The maximal case is more involved and in \cite[\S 7 and 8]{KSZ} we study in further generality maximal parabolic $G$-Higgs bundles, when the homogeneous space  ${G}/{H}$ is a Hermitian symmetric space of noncompact type, where $H\subset G$ is a maximal compact subgroup. In the special case $\text{SL}\left( 2,\mathbb{R} \right)=\text{Sp}\left( 2,\mathbb{R} \right)$, we have seen in Theorem 7.10 of \cite{KSZ} that for \emph{maximal} parabolic  $\text{SL}\left( 2,\mathbb{R} \right)$-Higgs bundles equipped with a \emph{fixed} parabolic structure, the connected components of the moduli space ${{\mathsf{\mathcal{M}}}_{par}^{\text{max}}}\left( \text{Sp}\left( 2,\mathbb{R} \right) \right)$ are parameterized by the square roots of $K\left( D \right)$, thus contributing to at least $2^{2g+s-1}$ maximal connected components, a number which comes from the number of the square roots of $K(D)$ as line $V$-bundles. This is also giving the number of the connected components for when the Toledo invariant achieves its minimum value $\tau =-\left( g-1+\frac{s}{2} \right)$. In conclusion, the total number of topological invariants is $2^{2g+s}+2^s(2g-3+s)$ adding up the maximal and the non-maximal cases for the Toledo invariant $\tau$; we thus have shown the following:

\begin{prop}\label{401}
Let $X$ be a smooth Riemann surface of genus $g$ and let $D$ be a reduced effective divisor of $s$ many points on $X$. The minimum number of connected components of the moduli space ${{\mathsf{\mathcal{M}}}_{par}}\left( \text{SL}\left( 2,\mathbb{R} \right) \right)$ of polystable parabolic $\text{SL}\left( 2,\mathbb{R} \right)$-Higgs bundles over the pair $(X, D)$ is $2^{2g+s}+2^s(2g-3+s)$.
\end{prop}

\subsection{Topological invariants in the case when $G=\text{GL}\left( 2,\mathbb{R} \right)$}

From the correspondence to Higgs $V$-bundles (see \S 3.2) a parabolic  $\text{GL}\left( 2,\mathbb{R} \right)$-Higgs bundle $(E, \Phi)$ corresponds to a Higgs $V$-bundle over the $V$-surface $M$ with structure group $\text{O}(2,\mathbb{C})$. We shall use the same notation $(E,\Phi)$ for the corresponding Higgs $V$-bundle. Thus, there is a real rank 2 orthogonal vector bundle $Y$, such that $E= Y\otimes \mathbb{C}$ and the Stiefel-Whitney classes of $Y$, $w_{i}=w_{i}(Y)\in H^i_V(M,\mathbb{Z}_2)$, for $i=1,2$, define appropriate topological invariants for the pairs $(E, \Phi)$. Note that $\text{rk}(H^1_V(M,\mathbb{Z}_2))=2g+s-1$ and $\text{rk}(H^2_V(M,\mathbb{Z}_2))=s$ (see \S 7.3 of \cite{KSZ}).

We may now classify as follows the subspaces of the moduli space of parabolic $\text{GL}\left( 2,\mathbb{R} \right)$-Higgs bundles with fixed values for these topological invariants:

\begin{enumerate}
\item[(1)] If $w_{1}\neq 0$, then $Y$ can be considered as a real bundle over the $V$-surface $M$ with structure group $\text{O}(2)$. Thus the first and second $V$-cohomology can be taken as the topological invariants for the real bundle $Y$, and equivalently for $E$. Thus we have $2^s (2^{2g+s-1}-1)$ many topological invariants in this case.

\item[(2)] If $w_{1}= 0$, then the structure group of $E$ reduces to $\text{SO}(2, \mathbb{C}) \subset \text{O}(2,\mathbb{C})$ and thus $E$ decomposes as $E=N \oplus N^{\vee}$, for a parabolic line bundle $N$. The semistability of the initial pair $(E, \Phi)$ implies that  $0\leq par\text{deg}(N) \leq g-1+\frac{s}{2}$.
    \begin{enumerate}
    \item[a.] If $0\leq par\text{deg}(N) <  g-1+\frac{s}{2}$, then for each fixed parabolic structure on the initial pair $(E, \Phi)$, there is an induced fixed parabolic structure on the line bundle $N$, and so there are $g-1+\frac{s}{2}$ many disjoint subspaces of $\mathcal{M}_{par}(\text{GL}\left( 2,\mathbb{R} \right))$ with fixed value of the topological invariant $par\deg(N)$. Now, allowing any parabolic weight on this parabolic line bundle $N$, we have a total of $2^{s}\cdot (g-1+\frac{s}{2})$ disjoint subspaces.

    \item[b.] If $par\deg(N)=g-1+\frac{s}{2}$, then the corresponding subspaces are parameterized by the square roots $L_0$ of the parabolic line bundle $K(D)$, thus distinguishing a total of $2^{2g+s-1}$ disjoint subspaces of $\mathcal{M}_{par}(\text{GL}\left( 2,\mathbb{R} \right))$.
    \end{enumerate}
\end{enumerate}

Summing up the number of disjoint subspaces for the different values of topological invariants analyzed above, we derive the following proposition:

\begin{prop}\label{402}
Let $X$ be a smooth Riemann surface of genus $g$ and let $D$ be a reduced effective divisor of $s$ many points on $X$.
The minimum number of connected components of the moduli space of parabolic $\text{GL}\left( 2,\mathbb{R} \right)$-Higgs bundles ${{\mathsf{\mathcal{M}}}_{par}}\left( \text{GL}\left( 2,\mathbb{R} \right) \right)$ is $2^s(2^{2g+s-1}-1)+2^{s}\cdot (g-1+\frac{s}{2})+2^{2g+s-1}$.
\end{prop}

\begin{rem}\label{403}
In view of Remark \ref{250}, considering the moduli space of $L$-twisted parabolic $\text{GL}\left( 2,\mathbb{R} \right)$-Higgs bundles with $L\cong (K(D))^2$ and the parabolic Cayley correspondence, then the above proposition implies the number of connected components of the moduli space of parabolic maximal $\text{Sp}\left( 4,\mathbb{R} \right)$-Higgs bundles as was computed in Theorem 7.9 of \cite{KSZ}. Note that in this case, since $L\cong (K(D))^2$ one has that $0\leq par\text{deg}(N) \leq 2g-2+s$; then the contribution to the number of disjoint subspaces from Cases (1) and (2)b above is the same, while one has a total of $2^{s}\cdot (2g-2+s)$ disjoint subspaces from (2)a.
\end{rem}

\subsection{Topological invariants in the case when $G=\text{PGL}\left( 2,\mathbb{R} \right)$}

The discussion of the topological invariants of parabolic $\text{PGL}(2,\mathbb{R})$-Higgs bundles is induced to the one for the $\text{GL}(2,\mathbb{R})$-case, further considering the equivalence relation and the parabolic structure of the parabolic line bundle $A$ in the definition.

Let $(E, \Phi, A)$ be a polystable parabolic $\text{PGL}\left( 2,\mathbb{R} \right)$-Higgs bundle. The stability condition for this triple (see Remark \ref{210}, (4)) implies that $par\deg(E)=0$ or $\frac{1}{2}$. Let $(U,\Phi)$ be a parabolic $\text{GL}(2,\mathbb{R})$-Higgs bundle, which is a representative of $(E,\Phi,A)$. By the discussion in \S 3.2 and \S 3.3, let $X_\eta$ be the spectral covering of $X$ determined by the parabolic Higgs field $\Phi$. If we forget the parabolic structure of $U$, we can define $L$ to be the corresponding line bundle over $X_\eta$. In the parabolic case it may not be a one-to-one correspondence between parabolic line bundles $L$ over $X_\eta$ and the parabolic Higgs bundle $(U,\Phi)$. Fortunately however, in the case of $\text{GL}(2,\mathbb{R})$ the correspondence is one-to-one on a particular subvariety (see Proposition \ref{701}). In other words, there are $2^{2s}$ many choices for the parabolic structures of the corresponding parabolic line bundle $L$, each of which corresponds to a unique $\text{GL}(2,\mathbb{R})$-Higgs bundle $(U,\Phi)$.

Now we are going to discuss the non-degenerate symmetric bilinear pairing $\langle ,\rangle: E \times E \rightarrow A$. If we fix the parabolic structure of $A$, the pairing $\langle ,\rangle$ will determine the parabolic structure of $U$ as follows
\begin{enumerate}
\item If the parabolic structure of $A$ over $x$ is $0 \subseteq A|_x$ with weight $\frac{1}{2}$, then the parabolic structure of $U$ over $x$ is of \emph{Type 1} (see \S 3.3).
\item If the parabolic structure of $A$ over $x$ is $0 \subseteq A|_x$ with weight $0$, then the parabolic structure of $U$ over $x$ is of \emph{Type 2}.
\end{enumerate}
Thus if we fix the parabolic structure of $A$, we have $2^{s}$ many choices for the parabolic structure of $U$. In the definition of a $\text{PGL}(2,\mathbb{R})$-Higgs bundle (Definition \ref{214}) two representatives $(U_1,\Phi_1)$, $(U_2,\Phi_2)$ are equivalent, if there is a parabolic line bundle $B$ such that $(U_1,\Phi_1)=(U_2 \otimes B, \Phi_2 \otimes {\rm Id})$. Under this equivalence all of the possible choices of $U$ are equivalent. Finally, note that there is a one-to-one correspondence between the parabolic structures and the second $V$-cohomology $H^2(M,\mathbb{Z}_2)$. Thus the second $V$-cohomology $H^2(M,\mathbb{Z}_2)$ does not play a role in the case of parabolic $\text{PGL}(2,\mathbb{R})$-Higgs bundles.

There are $2^s$ many choices for the parabolic structure of $A$. If we fix the parabolic degree of $A$ (or equivalently fix the parabolic degree of $U$), there are $2^{s-1}$ choices for the parabolic structure of $A$, each of which determines the monodromy around the punctures in $D$. In other words, consider the first $V$-cohomology in terms of the $V$-fundamental group
\begin{align*}
H_V^1(M,\mathbb{Z}_2)={\rm Hom}(\pi_V^1(M),\mathbb{Z}_2),
\end{align*}
where the $V$-fundamental group $\pi_V^1(M)$ is
\begin{align*}
\pi_V^1(M)=\{a_i,b_i,c_j, 1 \leq i \leq g, 1 \leq j \leq s | [a_1,b_1] \dots [a_g,b_g]c_1 \dots c_s=1, c_j^2=1, 1 \leq j \leq s\}
\end{align*}
and $c_j$ is the generator of the loop around the puncture $x_j$ in $D$. The parabolic structure of $A$ over $x_j$ determines the value of $c_j$.

Now we are ready to discuss the topological invariants. With the same notation as above, let $(U,\Phi)$ be a representative of $(E, \Phi, A)$. Let $U \otimes A^{\vee}=Y \otimes \mathbb{C}$, where $Y$ is the corresponding real bundle. Generally speaking, we consider the real bundle $Y$ as a real bundle over the $V$-surface $M$. This gives a well-defined element in the first $V$-cohomology $H^1_V(M,\mathbb{Z}_2)$. Denote by $w_1$ the first $V$-cohomology of $Y$. Abusing the language, we say that $w_1$ is the topological invariant of $U$.

\begin{enumerate}
\item[(1)] By the rank of $H_V^1(M,\mathbb{Z}_2)$, there are $2^{2g+s-1}$ many non-isomorphic choices for the topological invariants of $w_1$, i.e. for the choices of $(E,\Phi,A)$. As we discussed above, the real bundle $Y$ is considered as a real bundle over the $V$-surface $M$. Note that if we tensor an appropriate parabolic line bundle $B$, then $Y$ can be considered as a real bundle over $X$. Thus we have $2^{2g}$ many choices of the topological invariants for each parabolic structure of $A$. In other words, we go back to the discussion of non-parabolic $\text{GL}(2,\mathbb{R})$-Higgs bundles after we fix the parabolic structure of $A$ (see \cite{BarScha}, \S 5).

Now we fix the parabolic structure of $A$. We can consider $w_1$ as an element in $H^1(X,\mathbb{Z}_2)$ by forgetting the generators $c_j$, $1 \leq j \leq s$. If $w_1 \neq 0$ (as an element in $H^1(X,\mathbb{Z}_2))$, we show that there are $2^{2g}-1$ many topological invariants. If we only fix the parabolic degree of $A$ ($0$ or $\frac{1}{2}$), we have $2^{s-1}$ many choices for the parabolic structure of $A$. Thus we have $2^{s-1}(2^{2g}-1)=2^{2g+s-1}-2^{s-1}$ many choices for the topological invariants.

    In conclusion, if $w_1$ is nontrivial as an element in $H^1(X,\mathbb{Z}_2)$, there are $2^s(2^{2g}-1)$ many disjoint subspaces of the moduli space $\mathcal{M}_{par}(\text{PGL}\left( 2,\mathbb{R} \right))$ consisting of polystable parabolic $\text{PGL}\left( 2,\mathbb{R} \right)$-Higgs bundles with fixed first $V$-cohomology.

\item[(2)] If $w_1=0$, the parabolic bundle $U \otimes A^{\vee}$ can be decomposed as a direct sum of line bundles $U \otimes A^{\vee}=N_1 \oplus N_2$ with a similar discussion as in the case of $\text{GL}(2,\mathbb{R})$. By Remark \ref{210}, $N_1$ and $N_2$ must be parabolic dual to each other. Thus we rewrite the decomposition of $U \otimes A^{\vee}$ as $N \oplus N^{\vee}$. As was discussed above, given the parabolic structure of $A$ we have two choices for the parabolic structure of $E$ up to equivalence, and the parabolic structure of $N$ is uniquely determined by $E$. In conclusion, we have $2^s$ many choices for the parabolic structure of $A$ and $2^{s+1}$ many choices for the parabolic structure of $N$. If we fix the parabolic structure of $N$, then we go back to the discussion of Theorem 6.8 in \cite{BarScha} in the non-parabolic case. The only difference is that the authors in \cite{BarScha} discussed the parity (odd or even) of the degree of $A$ (as a line bundle), while we discuss whether the parabolic degree is an integer or not. In any case, there are $2^s\frac{2g-2+s}{2}$ many choices. In conclusion, there are $2^{s-1}(2g-2+s)$ many choices of topological invariants.
\end{enumerate}

Summing up the number of disjoint subspaces for the different values of topological invariants analyzed above, and multiplying by 2 since there are two different choices for the value of the parabolic degree of $E$, we derive the following proposition:

\begin{prop}\label{404}
Let $X$ be a smooth Riemann surface of genus $g$ and let $D$ be a reduced effective divisor of $s$ many points on $X$.
The minimum number of connected components of the moduli space ${{\mathsf{\mathcal{M}}}_{par}}\left( \text{PGL}\left( 2,\mathbb{R} \right) \right)$ of polystable parabolic $\text{PGL}\left( 2,\mathbb{R} \right)$-Higgs bundles over $(X, D)$ is $2^{2g+s}+2^s(2g-3+s)$.
\end{prop}

\section{Monodromy action and Lattice of a $V$-surface}

\subsection{Regular Locus $H_{reg}$ of the Parabolic Hitchin Map}
Recall that by $H_{reg}$ we denoted the regular locus of the parabolic Hitchin map in $H^0(X,K(D)^2)$. We say an element $a \in H^0(X,K(D)^2)$ is \emph{regular}, if and only if it only has simple zeros and the intersection of its zero set and $D$ is empty. In this section we study the fundamental group $\pi_1(H_{reg},a_o)$ of the regular locus as well as the fundamental group $\pi_1(\mathbb{P}H_{reg},a_0)$ of the projective space of the regular locus $\mathbb{P}H_{reg}$,  where, by abuse of notation, $a_o$ is a fixed point in $H_{reg}$ and in $\mathbb{P}H_{reg}$. We shall use the notation $\pi_1(H_{reg})$ and $\pi_1(\mathbb{P}H_{reg})$ for the fundamental group without referring to the base point. For proving most of the results included in this section, we extend the techniques from \cite{BarScha}, \cite{Cope} and \cite{Walk}, used in describing the monodromy in the non-parabolic case.

Let $X$ be a genus $g$ Riemann surface with a divisor $D$ containing $s$-many points. Denote by $X^{[n]}$ the space of $n$-tuples $(b_1,...,b_n) \in X^n$ of distinct un-ordered points in $X$. From this point on, we set $l:=\deg K(D)=2g-2+s$ to ease notation in the exposition.  Let $\rho : H_{reg} \rightarrow X^{[2l]}$ be the map taking sections to its zero set. We define the Abel map
\begin{gather*}
\mathcal{A} :  X^{[2l]}  \rightarrow  \text{Jac}(X)\\
D'  \rightarrow  |D'|\otimes K(D)^{-2},
\end{gather*}
where $D' \in X^{[2l]}$ is a pair of points in $X$. Note that $\text{Jac}(X) \cong H^0(X,K)^*/H_1(X,\mathbb{Z})$. Then, we have a natural isomorphism $H_1(\text{Jac}(X)) \cong H_1(X,\mathbb{Z})$. Now fix a point
\begin{align*}
b_o=\sum\limits_{i=1}^{2l}b_i \in X^{[2l]}.
\end{align*}
The composition
\begin{align*}
\pi_1(X^{[2l]}):=\pi_1(X^{[2l]},b_o) \xrightarrow{\mathcal{A}_*} \pi_1(\text{Jac}(M)) \rightarrow H_1(\text{Jac}(X)) \rightarrow H_1(X,\mathbb{Z})
\end{align*}
takes a braid in $X^{[2l]}$ to a union of homology classes of loops in $X$. Also, there is a natural inclusion map
\begin{align*}
H_1(X,\mathbb{Z}) \rightarrow H_1(X/D,\mathbb{Z}).
\end{align*}
Thus, we define the map
\begin{align*}
\nu :\pi_1(X^{[2l]}) \rightarrow H_1(X/D,\mathbb{Z})
\end{align*}
as the composition of the above maps.

\begin{lem}\label{501}
The kernel of the map $\nu$ is generated by transpositions of the points in $b_o$.
\end{lem}

\begin{proof}
The proof of the lemma is exactly the same as the one for Theorem 7 in \cite{Cope} and Theorem 4.2 in \cite{BarScha}. We give here only the construction of transpositions (swaps).\\
Recall that $b_o=b_1+b_2+\dots+b_{2l}$ is the fixed point in $X^{[2l]}$. Let $\gamma_{ij} :[0,1] \rightarrow \Sigma$ be an embedded path joining $b_i=\gamma(0)$ and $b_j=\gamma(1)$, where $i \neq j$ and $\gamma$ meets no other point of $b_o$. Let $D^2$  be the unit disc in $\mathbb{R}^2$ and $e : D^2 \rightarrow \Gamma$ an orientation preserving embedding such that $\gamma(t)=e(t-\frac{1}{2},0)$ and $e(D^2)$ contains no other points of $b_o$. Similarly, we define two other curves $\gamma^+_{ij}$ and $\gamma^-_{ij}$ by
\begin{align*}
\gamma^+_{ij}=e(t-\frac{1}{2},\text{sin}(\pi t)), \quad \gamma^-_{ij}=e(t-\frac{1}{2},-\text{sin}(\pi t)).
\end{align*}
Then we can define a loop $p_{\gamma_{ij}}$ based on $b_o$ as $p_{\gamma_{ij}}=\sum\limits_{1 \leq k \leq 2l}b_k(t)$, where $b_i(t)=\gamma^+_{ij}(t)$, $b_j(t)=\gamma^-_{ij}(t)$ and $b_k(t)=b_k$ for $k \neq i,j$, (see Figure 1). Clearly, the homotopy class $s_{\gamma_{ij}}=[p_{\gamma_{ij}}]$ in $X^{[2l]}$ does not depend on the choice of the loop. We call $s_{\gamma_{ij}}$ the \emph{swap} associated to $\gamma_{ij}$. The element of this form $p_{\gamma_{ij}}$ gives us a transposition of $b_i$ and $b_j$. The kernel of $\nu$ is generated by all $p_{\gamma}$ in this form.
\end{proof}

\vspace{5mm}
\begin{center}
  \includegraphics[width=0.6\linewidth,height=0.6\textheight,keepaspectratio]{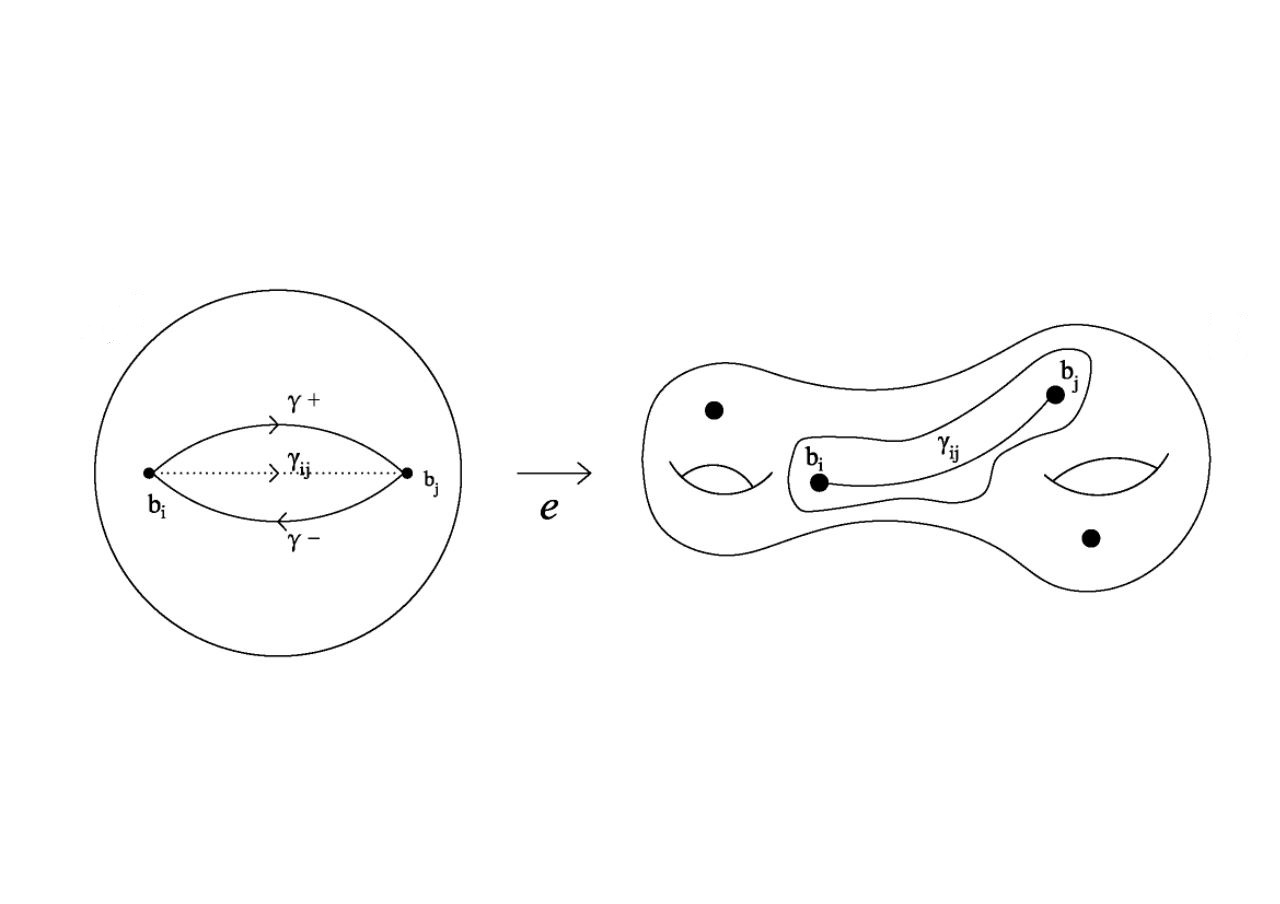}
   \captionof{figure}{A swap between $b_{i}$ and $b_{j}$ along the path $\gamma_{ij}$.}
\end{center}
\vspace{5mm}

Recall that the map $\rho : H_{reg} \rightarrow X^{[2l]}$ maps a regular section to its zero set. This map induces a map on the fundamental group
\begin{align*}
\rho_* : \pi_1(H_{reg}) \rightarrow \pi_1(X^{[2l]}).
\end{align*}
Hence, any transposition in $\pi_1(X^{[2l]})$ can be lifted to a transposition in $\pi_1(H_{reg})$. The following theorem implies that the image $\rho_*(\pi_1(H_{reg}))$ is generated by these transpositions.

\begin{thm}[Theorem 1.1, Theorem 1.2 in \cite{Walk}]\label{502}
When the genus $g \geq 2$, we have
\begin{align*}
\rho_* \pi_1(H_{reg})=\text{ker}(\nu),
\end{align*}
and the kernel is generated by transpositions of the points.
\end{thm}

J. Copeland in \cite{Cope} proved the above theorem in the case of hyperelliptic curves with quadratic differentials. K. Walker in \cite{Walk} generalized Copeland's idea and proved the result for any Riemann surface for both smooth fibers and singular fibers. We shall need the version of Walker's theorem in our case.

Although the image of the fundamental group $\rho_* \pi_1(H_{reg})$ is the kernel of $\nu$ and is generated by transpositions, the transpositions are not the only generators of the fundamental group $ \pi_1(H_{reg})$. Indeed, the fundamental group $\pi_1(H_{reg})$ is generated by both transpositions and the $\mathbb{C}^*$-action on the loop with respect to the point $a_o$.

\begin{defn}\label{503}
Let $\gamma : [0,1] \rightarrow H_{reg}$ be a loop in $H_{reg}$. We say that $\gamma$ is \emph{generated by the $\mathbb{C}^*$-action} with respect to a base point $a_0$, if $\gamma(t)=e^{2 \pi i t}a_0$.
\end{defn}

\begin{lem}[Proposition 4.4 in \cite{BarScha}]\label{504}
The fundamental group $\pi_1(H_{reg})$ is generated by the loop given by the $\mathbb{C}^*$-action on $H_{reg}$ together with transpositions.
\end{lem}

\subsection{Lattice and Monodromy Action}
Let $M$ be the $V$-surface defined by the data $(X,D,2)$. Denote by $M_\eta \rightarrow M$ the spectral covering with respect to $\eta \in H_{reg}$. Let $X$ and $X_\eta$ be the underlying space of $M$ and $M_\eta$ respectively. Similarly, there is an induced spectral covering $\pi: X_\eta \rightarrow X$ over the underlying surface. Let $$D_\eta=\{x'_i,x''_i, 1 \leq i \leq s\}$$ the inverse image of the divisor $D \in X$, where $x'_i$ and $x''_i$ are the pre-images of $x_i$. Under the parabolic BNR correspondence \cite{KSZ2}, we know that the fiber is some subvariety of the $V$-Picard group $\text{Pic}_V(M_\eta)$.

For a surface $X$, we follow notation from \cite{BarScha}. Let $\Lambda_X=H^1(X,\mathbb{Z})$ be the lattice of $X$ and let $\Lambda_P$ be the lattice of the Prym variety defined as the kernel of $\Lambda_{X_\eta} \rightarrow \Lambda_{X}$. Similarly, we next define the lattices in terms of a $V$-surface. Define the lattice $\Lambda_S:=H^1_V(S,\mathbb{Z})$ to be the first $V$-cohomology of $S$, where $S$ is a $V$-surface. The covering map $\pi$ induces the push-forward map $\pi_*: \Lambda_{M_\eta} \rightarrow \Lambda_M$ and we define the lattice $\Lambda_{P_V}$ to be the kernel of $\pi_*:\Lambda_{M_\eta}\rightarrow \Lambda_M$. In other words, the lattice $\Lambda_{P_V}$ corresponds to the Prym variety $\text{Prym}(M_\eta,M)$ of the covering $\pi: M_\eta \rightarrow M$.

Let $S$ be a $V$-surface such that the local chart around $x \in D$ is isomorphic to $\mathbb{D}^2/ \mathbb{Z}_{m}$. We have the following presentations of fundamental groups:
\begin{align*}
& \pi_1(S)=\{a_1,\dots,a_g,b_1,\dots,b_g,c_1,\dots,c_s|[a_1,b_1]\dots[a_g,b_g]c_1\dots c_s=1, c_i^{m}=1\},\\
& \pi_1(|S|/D)=\{a_1,\dots,a_g,b_1,\dots,b_g,c_1,\dots,c_s|[a_1,b_1]\dots[a_g,b_g]c_1\dots c_g=1\},
\end{align*}
where $|S|$ is the underlying surface of the $V$-surface $S$ and $D$ is the fixed divisor. There is a natural surjection $\pi_1(|S|\backslash D) \rightarrow \pi_1(S)$, which induces the following map in cohomology
\begin{align*}
H^1_V(S,\mathbb{Z}) \longrightarrow H^1(|S| \backslash D,\mathbb{Z}).
\end{align*}
If $m=2$, these two cohomology groups are isomorphic in $\mathbb{Z}_2$-coefficients. In other words, the $2$-torsion points of the lattice $\Lambda_S$ is isomorphic to the $2$-torsion points of the lattice $\Lambda_{|S|\backslash D}$. Therefore, we can consider $\Lambda_S$ as the lattice for the noncompact surface $|S| \backslash D$ in this special case. This property gives us a way to study the monodromy action in the case of $V$-surfaces. In general, if $m \neq 2$, this property is not apparent.

\vspace{5mm}
\begin{center}
 \includegraphics[width=0.6\linewidth,height=0.6\textheight,keepaspectratio]{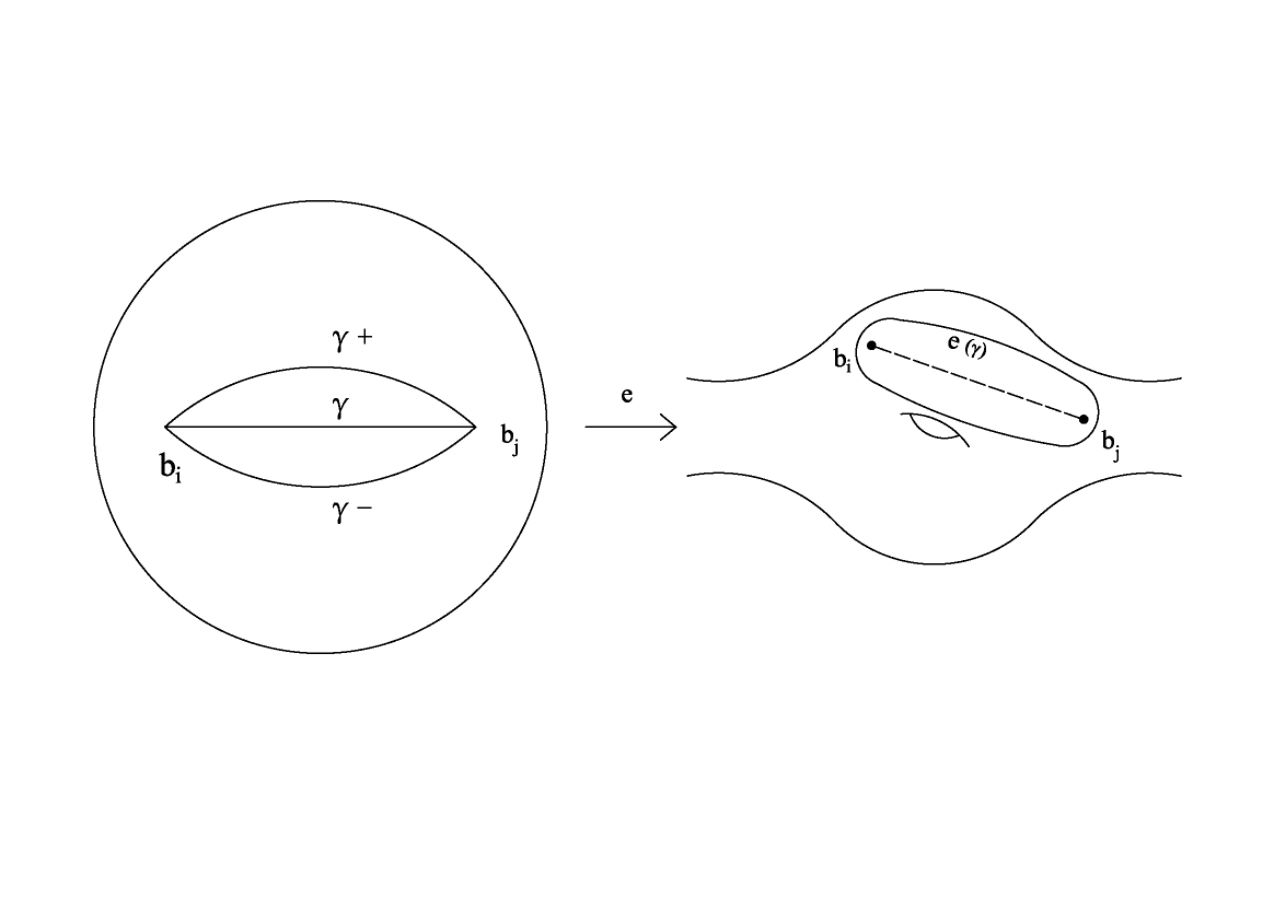}
   \captionof{figure}{A swap between $b_{i}$ and $b_{j}$ along the path $\gamma_{ij}$ for a surface with boundary components.}
\end{center}
\vspace{5mm}

\begin{thm}\label{505}
Let $l_\gamma$ be the embedded loop in $M_\eta \backslash D_\eta$ given by the preimage $\pi^{-1}(\gamma)$ of $\gamma$. The monodromy action $\rho_{*}(s_\gamma)$ gives an automorphism of $H^1(M_\eta \backslash D_\eta,\mathbb{Z})$, which is induced by a Dehn twist of $X_\eta$ around $l_\gamma$.
\end{thm}

\begin{proof}
The proof of this theorem is similar to the proof for Theorem 4.8 in \cite{BarScha}. The only difference is that we are working with the $V$-surface, in other words, a non-compact surface.\\
By construction of the $V$-surface $M_\eta$, a loop $\gamma$ on $M \backslash D$ is also a loop on its underlying surface $X$ with $s$-many punctures. Hence, we reduce the structure of the $V$-surface $M_\eta$ to a non-compact surface $X_\eta \backslash D_\eta$.\\
Let $\gamma=\gamma_{ij}$ for some $1 \leq i,j \leq 2l$. Recall that $\gamma_{ij}$ is an embedded path in $X \backslash D$ joining branch points $b_i,b_j$ and avoiding all other branch points and punctures. The transposition $p_{\gamma_{ij}}(t)$ associated to $\gamma_{ij}$ provides a loop $b_{\gamma_{ij}}(t)$ based at $b_o$. We choose an embedding $e: D^2 \rightarrow X$ of the unit disc $D^2$ into $X$ containing all branch points $b_1,\dots,b_{2l}$, all punctures and the loop $b_{\gamma{ij}}$. Let $\pi_t : X_t \rightarrow X$, $t \in[0,1]$ be the resulting family of branched double coverings of $X$. Since we make no change to the double cover outside the image $e(D^2)$, we only have to understand the family $S_t|_{\pi^{-1}_t (e(D^2))}$. It is well-known from Picard-Lefschetz theory \cite{DHM} that the monodromy is given by a Dehn twist of $S_t|_{\pi^{-1}_t (e(D^2))}$ around the cycle $l_{\gamma_{ij}}$.\\
More precisely, let $[h_\gamma]$ be the associated homology class in $H_1(M_\eta \backslash D_{\eta},\mathbb{Z})$. Clearly, $\pi_*([h_\gamma])=0$. Let $c_\gamma \in H^1(M_\eta \backslash D_\eta,\mathbb{Z})$ be its Poincar\'{e} dual class. Then, $c_\gamma \in \Lambda_P$. The Dehn twist of $X_\eta$ around $h_\gamma$ acts on $H^1(X_\eta \backslash D_\eta,\mathbb{Z})$ as a Picard-Lefschetz transformation. Thus the monodromy action of the loop associated to $\gamma$ is
\begin{align*}
\gamma \cdot x= \rho_*(s_\gamma) x = x+\langle c_\gamma, x \rangle c_\gamma.
\end{align*}
This completes the proof of the theorem.
\end{proof}

\subsection{Monodromy Action on $2$-Torsion Points on Lattices}
For a closed Riemann surface $X$, let $\Lambda_X[2]$ be the set of $2$-torsion points in $\Lambda_X$. An element in the Prym variety $\Lambda_P [2]$ of the covering $X_\eta \rightarrow X$ corresponds to a line bundle $L$ on $X_\eta$ such that $L^2=\mathcal{O}$ and $\tau^* L \cong L^{-1}$, where $\tau : X_\eta \rightarrow X_\eta$ is the involution map. We define the lattice $\widetilde{\Lambda}_P[2]$ to be the set of pairs $(L,\widetilde{\tau})$ such that $L \in \text{Prym}(X_\eta,X)$ satisfies $L^2 \cong \mathcal{O}$ and $\widetilde{\tau}:L \rightarrow L$ is an involution of the line bundle with respect to the involution of the spectral covering $\tau$. The lattice $\widetilde{\Lambda}_P[2]$ can be considered (see \cite{BarScha} for details) as a $\mathbb{Z}_2$-extension of $\Lambda_P[2]$
\begin{align*}
0 \longrightarrow \mathbb{Z}_2 \longrightarrow \widetilde{\Lambda}_{P}[2] \longrightarrow \Lambda_{P}[2] \longrightarrow 0.
\end{align*}

Now we move to the case of a $V$-surface. Let $\Lambda_{S}[2]=\Lambda_S \otimes_{\mathbb{Z}}\mathbb{Z}_2$ be the set of $2$-torsion points in $\Lambda_S$, where $S$ is a $V$-surface. Similarly, $\Lambda_{P_V}[2]$ is the $2$-torsion points of the Prym variety $\Lambda_{P_V}$ of the covering $\pi:M_\eta \rightarrow M$, which consists line bundles $L \in \text{Jac}(M_\eta)$ satisfying $L^2 \cong \mathcal{O}$ and $\tau^*(L) \cong L^{-1}$, where $\tau:M_\eta \rightarrow M_\eta$ is the involution of the $V$-surface $M_\eta$. We use the same notation $\tau$ for the involutions of $X_\eta$ and $M_\eta$. Similarly, The lattice $\widetilde{\Lambda}_{P_V}[2]$ is defined as the pairs $(L,\widetilde{\tau})$ such that $L$ is a $2$-torsion point in $\Lambda_{P_V}$ and $\widetilde{\tau}$ is an involution of $L$ with respect to the involution $\tau: M_\eta \rightarrow M$ of the covering. The lattice $\widetilde{\Lambda}_{P_V}[2]$ can be then considered as a $\mathbb{Z}_2$-extension of $\Lambda_{P_V}[2]$
\begin{align*}
0 \longrightarrow \mathbb{Z}_2 \longrightarrow \widetilde{\Lambda}_{P_V}[2] \longrightarrow \Lambda_{P_V}[2] \longrightarrow 0.
\end{align*}

Recall that $l$ denotes the degree of $K(D)$. Therefore the covering has $2l$ many branch points (ramification points) and denote by $B=\{b_1,...,b_{2l}\}$ the set of branch points. Since there is a one-to-one correspondence between branch points and ramification points, we keep the same notation $B$ for the ramifications. Note that $B \cap D = \emptyset$ by the definition of regular locus and also $B \cap D_\eta = \emptyset$.

Now let us consider the line bundle $L$ over $M_\eta$ as a $\mathbb{Z}_2$-local system together with an isomorphism $\tau^*: L \rightarrow L$. If $b_{i}$ is the ramification point, then $\tau^*$ maps $L_{b_i}$ to itself, acting either as $1$ or $-1$. Let $\varepsilon$ be the number such that $\tau^*$ acts on $L_{b_i}$ by $(-1)^{\varepsilon}$. Repeating this procedure for points $x \in D_\eta$, then the pair $(L,\widetilde{\tau})$ gives us an element
\begin{align*}
\varepsilon(L,\widetilde{\tau}^*)=\sum\limits_{i=1}^{l}\varepsilon_i b_i + \sum\limits_{i=1}^{l}\varepsilon'_i (x'_i+x''_i) \in \mathbb{Z}_2 (B+D_\eta).
\end{align*}
Note that the coefficients of $x'_i$ and $x''_i$ are the same by the condition $\tau^*L=L$. The restriction of $L$ to $(M_\eta \backslash D_\eta) \backslash B$ descends to a local system $L'$ on $(M \backslash D) \backslash B$. Clearly, $$\varepsilon(L,-\widetilde{\tau})=\varepsilon(L,\widetilde{\tau})+b_o+x_o,$$ where $b_o=b_1+\dots+b_{2l}$ and $x_0=x_1+\dots+x_s$. This implies a map
\begin{align*}
\varepsilon: \Lambda_{P_V}[2] \longrightarrow (\mathbb{Z}_2 B)^{ev},
\end{align*}
where $\mathbb{Z}_2 B$ is the $\mathbb{Z}_2$-coefficient vector space with basis $B=\{b_1,\dots, b_l\}$ and $(\mathbb{Z}_2 B)^{ev}$ denotes the subspace of $\mathbb{Z}_2 B$ with elements $\sum\limits_{i=1}^l \epsilon_i b_i$ such that $\sum\limits_{i=1}^l \epsilon_i \text{ mod } 2 =0$.

\begin{lem}\label{506}
The following three sequences for the lattice $\Lambda_{P_V}$ are exact:
\begin{align*}
& 0 \longrightarrow \underbrace{\Lambda_{M}[2]}_{2g+s-1} \longrightarrow \underbrace{\Lambda_{P_V} [2]}_{(2g+s-1)+(2l-1)} \xrightarrow{\varepsilon}  (\mathbb{Z}_2 B)^{ev} \longrightarrow 0,\\
& 0 \longrightarrow \underbrace{\Lambda_{P}[2]}_{2g-2+2l} \longrightarrow \underbrace{\Lambda_{P_V} [2]}_{(2g-2+2l)+s} \rightarrow  \bigoplus_{i=1}^s \mathbb{Z}_2 \longrightarrow 0,\\
& 0 \longrightarrow \underbrace{\Lambda_{M}[2]}_{2g+s-1} \longrightarrow \underbrace{\Lambda_{M_\eta} [2]}_{2g_\eta +2s-1} \longrightarrow  \underbrace{\Lambda_{P_V}[2]}_{2g-2+s+2l} \rightarrow 0,
\end{align*}
where the number below each lattice resembles its rank and $g_\eta=2g-1+l$ is the genus of the underlying surface $X_\eta$.
\end{lem}

\begin{proof}
The genus of the spectral curve $X_\eta$ was computed in \S 3.1 to be $g_\eta=4g-3+s$. The first exact sequence follows from the discussion for the local system above. The second exact sequence comes from the exact sequences of the $V$-Picard group
\begin{align*}
& 0 \longrightarrow \text{Pic}(X) \longrightarrow \text{Pic}_V(M) \longrightarrow \bigoplus_{i=1}^s \mathbb{Z}_2 \longrightarrow 0,\\
& 0 \longrightarrow \text{Pic}(X_\eta) \longrightarrow \text{Pic}_V(M_\eta) \longrightarrow \bigoplus_{i=1}^{2s} \mathbb{Z}_2 \longrightarrow 0.\\
\end{align*}
Details can be found in \S 7 of \cite{KSZ2}. The proof of the third exact sequence is similar to Proposition 4.12 and 4.15 in \cite{BarScha}.
\end{proof}

\begin{rem}\label{507}
Based on the first exact sequence and the discussion of the local system above, we have the following decomposition of $\Lambda_M[2]$
\[\Lambda_{M}[2] \cong \left(\bigoplus_{i=1}^{s-1} \mathbb{Z}_2 \right) \oplus \Lambda_X[2]\]
and the following decomposition of $\Lambda_{P_V}[2]$
\begin{align*}
& \Lambda_{P_V}[2] \cong \left(\bigoplus_{i=1}^{s-1} \mathbb{Z}_2 \right) \oplus \Lambda_X[2] \oplus  (\mathbb{Z}_2 B)^{ev}.
\end{align*}
\end{rem}

\begin{lem}\label{508}
Let $M$ be the $V$-surface defined by the data $(X, D, 2)$. Let $\gamma$ be a path joining distinct branch points $b_i,b_j$ in the covering $M$ and $c_\gamma \in \Lambda_{P_V} [2]$ the corresponding cycle in $M_\eta$. Then,
\begin{align*}
\varepsilon(c_\gamma)=b_i+b_j.
\end{align*}
Conversely if $c$ is any element of $\Lambda_{P_V} [2]$ such that $\varepsilon(c)=b_i+b_j$, then there exists an embedded path $\gamma$ from $b_i$ to $b_j$ for which $c=c_\gamma$.
\end{lem}

\begin{proof}
Similar to the proof of Proposition 4.13 in \cite{BarScha}. The only difference is that we are working on the noncompact surface $X \backslash D$.
\end{proof}

Define a non-degenerate symmetric bilinear form $\left( \cdot ,\cdot  \right):\mathbb{Z}_2 B \otimes \mathbb{Z}_2 B \rightarrow \mathbb{Z}_2$ as
\begin{align*}
(b_i,b_j)=\begin{cases}
0 & \text{ if } i \neq j\\
1 & \text{ otherwise }.
\end{cases}
\end{align*}
Clearly, we have $(\mathbb{Z}_2 B)^{ev}=(b_o)^{\perp}$, where $(b_o)^{\perp}$ is the orthogonal complement of $b_o$. Hence, the restriction of the bilinear form $\left( \cdot ,\cdot  \right)$ to $(\mathbb{Z}_2 B)^{ev} / (b_o)= (b_o)^{\perp}/(b_o)$ is non-degenerate. Note that $(x,x)=(x,b_o)$ for any $x \in \mathbb{Z}_2 B$ and $(x,x)=0$ for any $x \in (\mathbb{Z}_2 B)^{ev}$. Denote by $\langle , \rangle$ the intersection form given by the pullback of $\left( \cdot ,\cdot  \right)$ under the map $\Lambda_{P_V}[2] \rightarrow (\mathbb{Z}_2 B)^{ev}$. In other words, the bilinear form $\langle ,\rangle$ is defined by
\begin{align*}
\langle (w_1,x_1,y_1),(w_2,x_2,y_2) \rangle=(y_1,y_2),
\end{align*}
for all $(w_1,x_1,y_1),(w_2,x_2,y_2) \in \Lambda_{P_V}[2] \cong \left(\bigoplus\limits_{i=1}^{s-1} \mathbb{Z}_2 \right) \oplus \Lambda_X[2] \oplus  (\mathbb{Z}_2 B)^{ev}$. The same approach can be applied to the decomposition $\Lambda_{P_V}[2] \cong \left( \bigoplus\limits_{i=1}^{s} \mathbb{Z}_2 \right) \oplus \Lambda_X[2] \oplus  (\mathbb{Z}_2 B)^{ev}/(b_o)$ under the map $\Lambda_{P_V}[2] \rightarrow (\mathbb{Z}_2 B)^{ev}/(b_o)$. We use the same notation $\langle ,\rangle$ for the bilinear form with respect to the decomposition of $\Lambda_{P_V}[2]$ from Remark \ref{507}.

Recall that there is natural intersection form $\langle ,\rangle$ on the Riemann surface $X$. This intersection form can be naturally extended to the punctured Riemann surface $M \backslash D \cong X \backslash D$, thus providing a well-defined intersection form $\langle , \rangle: \Lambda_M[2] \times \Lambda_M[2] \rightarrow \mathbb{Z}$. The existence of the intersection form $\langle , \rangle$ on $\Lambda_{M_\eta}[2]$ is now proved similarly as in Proposition 4.15 in \cite{BarScha} using our Lemma \ref{506}:

\begin{prop}\label{509}
There exists a decomposition of $\Lambda_{M_\eta}$
\begin{align*}
\Lambda_{M_\eta}[2] & \cong \Lambda_{M}[2] \oplus (\mathbb{Z}_2 B)^{ev} \oplus \Lambda_M[2] \\
& \cong \left(\bigoplus_{i=1}^{s-1} \mathbb{Z}_2 \right) \oplus \Lambda_{X}[2] \oplus (\mathbb{Z}_2 B)^{ev} \oplus \left(\bigoplus_{i=1}^{s-1} \mathbb{Z}_2 \right) \oplus \Lambda_{X}[2].
\end{align*}
The intersection form $\langle , \rangle$ on $\Lambda_{M_\eta}[2]$ is given by
\begin{align*}
\langle (x'_1,x''_1,y_1,z'_1,z''_1),(x'_2,x''_2,y_2,z'_2,z''_2) \rangle=\langle x''_1,z''_2 \rangle + (y_1,y_2)+\langle z''_1,x''_2 \rangle,
\end{align*}
for all $x'_1,x'_2,z'_1,z'_2 \in \left(\bigoplus\limits_{i=1}^{s-1} \mathbb{Z}_2 \right)$, $x''_1,x''_2,z''_1,z''_2 \in \Lambda_X[2]$, and $y_1,y_2 \in (\mathbb{Z}_2 B)^{ev}$.
\end{prop}

\begin{rem}\label{510}
The decomposition for $\Lambda_{M_\eta}$ can be written as
\begin{align*}
\Lambda_{M_\eta}[2] & \cong \Lambda_M[2] \oplus (\mathbb{Z}_2 B)^{ev} \oplus \Lambda_M[2] \\
& \cong \left(\bigoplus_{i=1}^{s-1} \mathbb{Z}_2 \right) \oplus \Lambda_{X}[2] \oplus (\mathbb{Z}_2 B)^{ev} \oplus \Lambda_{M}[2].
\end{align*}
We shall use the isomorphism $\Lambda_M[2] \cong \left(\bigoplus\limits_{i=1}^{s-1} \mathbb{Z}_2 \right) \oplus \Lambda_{X}[2]$ to highlight the first part $\left(\bigoplus\limits_{i=1}^{s-1} \mathbb{Z}_2 \right)$ when we discuss the monodromy action.
\end{rem}

For the rest of the section, we will study the monodromy action on $\Lambda_{M_\eta}[2]$ and $\Lambda_{P_V}[2]$ based on the decomposition of $\Lambda_{M_\eta}[2]$ and $\Lambda_{P_V}[2]$. Let $c$ be an element in the lattice $\Lambda_{M_\eta}[2]$. The corresponding Picard-Lefschetz transformation is
\begin{align*}
s_c(x)=x+ \langle c,x \rangle c.
\end{align*}
By Theorem \ref{505}, the monodromy action on $\Lambda_{M_\eta}[2]$ is generated by the involution $\tau$ together with $s_c$, where $c \in \Lambda_P[2]$ is of the form $c=(a,b_{ij},0)$, where $a \in \Lambda_{M}[2]$ and $b_{ij}=b_i+b_j \in (\mathbb{Z}_2 B)^{ev}$. Given a permutation $\omega \in S_{2l}$, there is a natural action on $B=\{b_1,...,b_{2l}\}$ by $\omega(b_i)=b_{\omega(i)}$. This action can be extended to $(\mathbb{Z}_2 B)^{ev}$. To simplify notation, set $s_{ij}:=s_{(0,b_{ij},0)}$. Under the monodromy action, $s_{ij}$ can be written in the form
\begin{align*}
s_{ij}=\begin{pmatrix}
I_{s-1} & 0 & 0 & 0 & 0\\
0 & I_{2g} & 0 & 0 & 0\\
0 & 0 & \sigma_{ij} & 0 & 0\\
0 & 0 & 0 & I_{2g} & 0\\
0 & 0 & 0 & 0 & I_{s-1}
\end{pmatrix},
\end{align*}
where $\sigma_{ij} \in S_{2l}$ is the transposition of $i$ and $j$. Now fix $x \in  \Lambda_X[2]$ and define another linear transformation $A_{ij}^x:=s_{b_{ij}} s_{b_{ij}+x}$. As a matrix, $A_{ij}^x$ can be written in the following form
\begin{align*}
A_{ij}^x=\begin{pmatrix}
I_{s-1} & 0 & 0 & 0 & 0\\
0 & I_{2g} & L_{ij}^x & S^x & 0\\
0 & 0 & I & (L_{ij}^x)^t & 0\\
0 & 0 & 0 & I_{2g} & 0\\
0 & 0 & 0 & 0 & I_{s-1}
\end{pmatrix},
\end{align*}
where $L_{ij}^x : (\mathbb{Z}_2 B)^{ev} \rightarrow \Lambda_{X}[2]$ is defined as $L_{ij}^x (b)=((b_{ij},b))x$. Here, $(L_{ij}^x)^t: \Lambda_X[2] \rightarrow \Lambda_M[2] \rightarrow (\mathbb{Z}_2 B)^{ev}$ is the adjoint map, and $S^x : \Lambda_X[2] \rightarrow \Lambda_X[2]$ is defined as $S^x(a)=\langle x,a \rangle x$. The monodromy action on $\Lambda_{M_\eta}[2]$ is then generated by $s_{ij}$ and $A_{ij}^x$.

\begin{rem}\label{511}
Based on the action defined above, we have $I_{2g+s-1}=\begin{pmatrix}I_{2g} & 0 \\ 0 & I_{s-1}
\end{pmatrix}$.
To ease notation, we may write the monodromy matrix $A_{ij}^x$ as \begin{align*}
A_{ij}^x=\begin{pmatrix}
I_{s-1} & 0 & 0 & 0 \\
0 & I_{2g} & L_{ij}^x & S^x \\
0 & 0 & I & (L_{ij}^x)^t  \\
0 & 0 & 0 & I_{2g+s-1} \\
\end{pmatrix}, \text{ or even, } A_{ij}^x=\begin{pmatrix}
I_{2g+s-1} & L_{ij}^x & S^x \\
0 & I & (L_{ij}^x)^t  \\
0 & 0 & I_{2g+s-1} \\
\end{pmatrix}.
\end{align*}
In fact, $A_{ij}^x$ is the matrix for the monodromy action for the parabolic $\text{GL}(2,\mathbb{R})$-Hitchin fibration as we shall prove next in Proposition \ref{512}. The upper left $3 \times 3$ block
\begin{align*}
\begin{pmatrix}
I_{s-1} & 0 & 0\\
0 & I_{2g} & (L_{ij}^x)\\
0 & 0 & I
\end{pmatrix}
\end{align*}
corresponds to the monodromy action on the parabolic $\text{SL}(2,\mathbb{R})$-Hitchin fibration (see Corollary \ref{513} and \ref{514}). On the other hand, the lower right $3 \times 3$ block
\begin{align*}
\begin{pmatrix}
I & 0 & 0\\
0 & I_{2g} & (L_{ij}^x)^t\\
0 & 0 & I_{s-1}
\end{pmatrix}
\end{align*}
corresponds to the $\text{PGL}(2,\mathbb{R})$-Hitchin fibration. The above monodromy action can be  naturally extended to the complex Lie groups $\text{GL}(2,\mathbb{R})$, $\text{SL}(2,\mathbb{R})$ and $\text{PGL}(2,\mathbb{R})$. Lastly, note that the monodromy action for the parabolic $\text{SL}(2,\mathbb{R})$-Hitchin fibration has a block $L_{ij}^x$, while its transpose matrix appears in the case of $\text{PGL}(2,\mathbb{R})$. Thus, there seems to be a ``duality" property undermining the monodromy action in the case of these two groups; we shall comment again on this phenomenon in Remark \ref{707} later on.
\end{rem}

\begin{prop}\label{512}
Let $G \subseteq \text{GL}(\Lambda_{M_\eta}[2])$ be the group generated by the monodromy action on $\Lambda_{M_\eta}[2]$.  Then $G$ is isomorphic to a semi-direct product $G \cong S_{2l} \ltimes H$, where $S_{2l}$ is isomorphic to the symmetric group generated by the transpositions $s_{ij}$ and $H$ is generated by the transformations $A_{ij}^x$, $x \in \Lambda_X[2]$, $1 \leq i<j \leq 2l$, for $l=2g-2+s$. The action of $\omega \in S_{2l}$ on $H$ is defined as
\begin{align*}
\omega(A_{ij}^x)=A^x_{\omega(i)\omega(j)}.
\end{align*}
\end{prop}

\begin{proof}
The proof is similar to that of Proposition 4.19 in \cite{BarScha} and Theorem 14 in \cite{Schap}.
\end{proof}

\begin{cor}\label{513}
Let $G_P \subseteq \text{GL}(\Lambda_{P_V}[2])$ be the group generated by the monodromy action on $\Lambda_{P_V}[2]$. Then the element in $G_P$ can be written in the form
\begin{align*}
\begin{pmatrix}
I_{s-1} & 0 & 0\\
0 & I_{2g} & A\\
0 & 0 & \omega
\end{pmatrix},
\end{align*}
where $\omega \in S_{2l}$ is the permutation and $A$ is an endomorphism $(\mathbb{Z}_2 B)^{ev} \rightarrow \Lambda_X[2]$.
\end{cor}

Note that the lattice $\widetilde{\Lambda}_{P_V}[2]$ is a $\mathbb{Z}_2$-extension of $\Lambda_{P_V}[2]$. As vector spaces, it holds that $\widetilde{\Lambda}_{P_V}[2] \cong \Lambda_{P_V}[2] \oplus \Lambda_{P_V}[2]$. Therefore the monodromy action on $\widetilde{\Lambda}_{P_V}[2]$ is the same as the monodromy action on $\Lambda_{P_V}[2]$ from Corollary \ref{513} when restricted to each summand.

\begin{cor}\label{514}
Let $G_P \subseteq \text{GL}(\widetilde{\Lambda}_{P_V}[2])$ be the group generated by the monodromy action on $\widetilde{\Lambda}_{P_V}[2]$. Let $G_{P_1}$ and $G_{P_2}$ be the groups generated by the monodromy action on the first and second summand of $\widetilde{\Lambda}_{P_V}[2] \cong \Lambda_{P_V}[2] \oplus \Lambda_{P_V}[2]$ respectively. Thus, the monodromy group $G_P$ is generated by $G_{P_1}$ and $G_{P_2}$.
\end{cor}

\section{Monodromy and parabolic $\text{SL}(2,\mathbb{R})$-Higgs Bundles}
In this section we apply the calculations for the monodromy described in the previous section in order to find an exact number of connected components for the moduli space of polystable parabolic $\text{SL}(2,\mathbb{R})$-Higgs bundles ${\mathsf{\mathcal{M}}}_{par}\left( \text{SL}\left( 2,\mathbb{R} \right) \right)$. Since there is a one-to-one correspondence between parabolic Higgs bundles and Higgs $V$-bundles, all parabolic Higgs bundles we mention in the proofs of this section should be viewed as their Higgs $V$-bundle counterparts.

\begin{prop}\label{601}
The fiber of the parabolic Hitchin map ${{\mathsf{\mathcal{M}}}_{par}}\left( \text{SL}\left( 2,\mathbb{R} \right) \right)\to \mathsf{\mathcal{H}}$ with respect to a point $a_0$ in the regular locus $H_{reg}$ is the space $\widetilde{\Lambda}_{P_V}[2]$.
\end{prop}

\begin{proof}
Let $(E,\Phi)$ be a parabolic $\text{SL}(2,\mathbb{C})$-Higgs bundle with parabolic structure $\alpha$. Let $M_\eta\rightarrow M$ be the corresponding spectral covering of the $V$-surface $M$, and let $\tau : M_\eta \rightarrow M_\eta$ be the involution of the covering $V$-surface $M_\eta$. Under the BNR correspondence for $V$-surfaces (see \cite{KSZ2}), denote by $L$ the corresponding line $V$-bundle over $M_\eta$, i.e. $\pi_*(L)=E$. Because this version of the BNR correspondence over $V$-surfaces may not be a one-to-one correspondence, the choice of $L$ is in general not unique. Fixing a square root $K(D)^{\frac{1}{2}}$ of $K(D)$, we write $L=L_0 \otimes \pi^*(K(D)^{\frac{1}{2}})$. Clearly, we have $\tau^* L_0 \cong L_0^{\vee}$. This means that $L_0$ is a well-defined element in the Prym variety $\text{Prym}(M_\eta,M)$. Also, note that the condition $\tau^* L_0 \cong L_0^{\vee}$ implies that the parabolic structures of $L_0$ over the pre-images of $x \in D$ should be the same. In other words, the parabolic structure of $E$ over $x \in D$ is a trivial filtration with weight $0$ or $\frac{1}{2}$. In this case, the parabolic BNR correspondence is in fact a one-to-one correspondence.\\
Let $(L_0,\widetilde{\tau})$ be a $2$-torsion point in $\widetilde{\Lambda}_{P_V}[2]$. Let $(E,\Phi)$ be the corresponding parabolic Higgs bundle defined as above. Since the parabolic structure $\alpha$ of $E$ could only be a trivial filtration with weight $0$ or $\frac{1}{2}$ and the parabolic BNR correspondence over a $V$-surface is unique for a trivial filtration, thus the choice of $L$ defined above is unique. As a $2$-torsion point, we have $L_0^2=\mathcal{O}$. Therefore $\tau^*L_0 \cong L_0$. The condition $L_0^2=\mathcal{O}$ also determines an orthogonal structure on $E$. The involution $\widetilde{\tau}:L_0 \rightarrow L_0$ determines an involution of $E$. We use the same notation $\widetilde{\tau}$ for the involution of $E$. With respect to the involution $\widetilde{\tau}$, we can decompose $E=N_+ \oplus N_-$ into eigenspaces with eigenvalue $+1$ and $-1$ respectively. Together with the orthogonal structure on $E$, $N_+$ and $N_-$ are dual to each other. The $V$-Higgs field $\Phi$ is induced by the tautological section $\lambda: L \rightarrow L \otimes \pi^*(K(D))$. Since $\tau^*(\lambda)=-\lambda$, then $\Phi$ is a skew-symmetric matrix, which can be written as
$
\begin{pmatrix}
0 & \beta \\
\gamma & 0
\end{pmatrix}
$, where $\beta \in \text{Hom}(N_-,N_+ \otimes K(D))$, $\gamma \in \text{Hom}(N_+,N_- \otimes K(D))$. The discussion above therefore provides a parabolic $\text{SL}(2,\mathbb{R})$-Higgs bundle $(N=N_+,\beta,\gamma)$.\\
Conversely, given a parabolic $\text{SL}(2,\mathbb{R})$-Higgs bundle $(N=N_+,\beta,\gamma)$, let $E:=N \oplus N^{\vee}$ and let $L_0$ be the associated parabolic line bundle as defined above. The parabolic structure of $N$ uniquely determines the parabolic structure of $L_0$. The orthogonal structure on $E$ implies that $L_0^2=\mathcal{O}$ and so $L_0$ is a $2$-torsion point. Also, the decomposition $E=N \oplus N^{\vee}$ determines an involution $\widetilde{\tau}$ of $E$ such that $N$ and $N^{\vee}$ are eigenspaces with eigenvalues $1$ and $-1$ respectively. This involution $\widetilde{\tau}$ induces an involution of the line bundle $L_0$, and we use the same notation $\widetilde{\tau}$ for the involution of $L_0$. In conclusion, a parabolic $\text{SL}(2,\mathbb{R})$-Higgs bundle $(N=N_+,\beta,\gamma)$ gives a well-defined element $(L_0,\widetilde{\tau})$ in $\widetilde{\Lambda}_{P_V}[2]$.
\end{proof}

\begin{prop}\label{602}
The number of orbits under the monodromy action on the Hitchin fibration $${{\mathsf{\mathcal{M}}}_{par}}\left( \text{SL}\left( 2,\mathbb{R} \right) \right)\to \mathsf{\mathcal{H}}$$ is $2^{2g+s}+2^s(l-1)$, where $l = par\deg K(D) = 2g-2+s$. It is also the maximal number of connected components of the moduli space ${{\mathsf{\mathcal{M}}}_{par}}\left( \text{SL}\left( 2,\mathbb{R} \right) \right)$.
\end{prop}

\begin{proof}
From Proposition \ref{601}, we know that in the case of parabolic $\text{SL}(2,\mathbb{R})$-Higgs bundles, the real points of a regular fiber are given by the lattice $\widetilde{\Lambda}_{P_V}[2]$. Moreover, $\widetilde{\Lambda}_{P_V}[2]$ is a $\mathbb{Z}_2$ extension of $\Lambda_{P_V}[2]$. Thus we first study the monodromy action on the lattice $\Lambda_{P_V}[2]$. From Remark \ref{507}, we can write $\Lambda_{P_V}[2]$ as
\begin{align*}
\Lambda_{P_V}[2] \cong \left(\bigoplus_{i=1}^{s-1} \mathbb{Z}_2 \right) \oplus \Lambda_X[2] \oplus  (\mathbb{Z}_2 B)^{ev}.
\end{align*}
With respect to the monodromy action we studied in \S 5, the ``maximal" orbits are those of the form $(c,a,0)$, where $c \in \bigoplus_{i=1}^{s} \mathbb{Z}_2$ and $ a \in \Lambda_X[2]$. This provides $2^{2g+s-1}$ many orbits.\\
The ``non-maximal" orbits have representatives of the form $(c,a,b)$, where $b \neq b_o, 0$. The matrix of the monodromy action can be written as
\begin{align*}
\begin{pmatrix}
I_{s-1} & 0 & 0\\
0 & I_{2g} & A\\
0 & 0 & \omega
\end{pmatrix},
\end{align*}
where $\omega$ is a permutation in $S_{2l}$. Thus any element $(c,a,b)$ can be reduced to the element $(c,0,b)$. Under the action of the permutation group $S_{2l}$, there are $l-1$ many orbits for the space $(\mathbb{Z}_2 B)^{ev}$. Therefore the number of orbits in the ``non-maximal" case is $2^{s-1} \cdot (l-1)$.\\
Combining the maximal and non-maximal cases, the number of orbits under the monodromy action on $\Lambda_{P_V}[2]$ is $2^{2g+s-1}+2^{s-1}(l-1)$. As a $\mathbb{Z}_2$-extension of $\Lambda_{P_V}[2]$, the number of orbits in $\widetilde{\Lambda}_{P_V}[2]$ under the monodromy action is $2 \times (2^{2g+s-1}+2^{s-1}(l-1))$. This implies the result of the proposition.
\end{proof}

\begin{prop}\label{603}
Every connected component of the moduli space $\mathcal{M}_{par}(G)$ meets the fibration corresponding to the regular locus of the Hitchin base, where $G=\text{GL}(2,\mathbb{R})$, $\text{SL}(2,\mathbb{R})$ and $\text{PGL}(2,\mathbb{R})$.
\end{prop}

\begin{proof}
The proof is exactly the same as the one of Proposition 6.3 in \cite{BarScha}.
\end{proof}

Putting together Propositions \ref{401}, \ref{602} and \ref{603}, we obtain the main result of this article:

\begin{thm}\label{604}
Let $X$ be a smooth Riemann surface of genus $g$ and let $D$ be a reduced effective divisor of $s$ many points on $X$.
The exact number of connected components of the moduli space ${{\mathsf{\mathcal{M}}}_{par}}\left( \text{SL}\left( 2,\mathbb{R} \right) \right)$ of polystable parabolic $\text{SL}(2,\mathbb{R})$-Higgs bundles over the pair $(X, D)$ is $2^{2g+s}+2^s(2g-3+s)$.
\end{thm}

\section{Parabolic $\text{GL}\left( 2,\mathbb{R} \right)$ and $\text{PGL}\left( 2,\mathbb{R} \right)$-Higgs Bundles}

We now turn to the computation of the orbits of the monodromy action for the $\text{GL}(2,\mathbb{R})$ and $\text{PGL}(2,\mathbb{R})$-parabolic Hitchin system. First recall the modified parabolic BNR correspondence from \S 3.3. Let $(E,\Phi)$ be a parabolic Higgs bundle over $(X,D)$ and let $X_\eta$ be the corresponding spectral curve. The parabolic BNR correspondence gives the existence of a line bundle $L$ over $X_\eta$ such that $\pi_*(L)=E$ and the Higgs field $\Phi$ is induced by the tautological section $\lambda: L \rightarrow L \otimes \pi^*(K(D))$. We fix a square root $K(D)^{\frac{1}{2}}$ of $K(D)$ (as a line $V$-bundle) and write $L=L_0 \otimes \pi^*K(D)^{\frac{1}{2}}$. Thus a parabolic Higgs bundle $(E,\Phi)$ corresponds to a line $V$-bundle (parabolic line bundle) $L_0$. Also, note that the orthogonal structure gives $(E, \Phi) \cong (E^{\vee}, \Phi^*)$, which implies $L_0^2 =\mathcal{O}$ as parabolic line bundles. We first prove the following:

\begin{prop}\label{701}
Let $(E,\Phi)$ be a polystable parabolic $\text{GL}(2,\mathbb{R})$-Higgs bundle and denote by $L_0$ the corresponding line $V$-bundle over the spectral $V$-surface $M_\eta$ such that $E = \pi_*(L_0 \otimes \pi^*(K(D)^{\frac{1}{2}}))$. The fiber of the parabolic Hitchin map ${{\mathsf{\mathcal{M}}}_{par}}\left( \text{GL}\left( 2,\mathbb{R} \right) \right)\to \mathsf{\mathcal{H}}$ with respect to a point $a_0$ in the regular locus $H_{reg}$ is the space $\Lambda_{M_\eta}[2]$.
\end{prop}

\begin{proof}
Let $(E,\Phi)$ be a parabolic Higgs bundle. If we forget the parabolic structure, it corresponds to a unique line bundle $L_0 \in \text{Pic}(X_\eta)$ such that $L_0^2 = \mathcal{O}$. This result is proved in Proposition 5.5 of \cite{BarScha}. We now implement the parabolic structures in the study. We still focus on the parabolic structure of $E|_x$ and $L_0|_{x'}$, $L_0|_{x''}$, where $x$ is a point in $D$ and $x',x''$ are its pre-images. Recall the distinction of the cases for the weights in the filtration of the pre-image of each point $x\in D$ from \S 3.3 as of Type 1 and Type 2.
\begin{enumerate}
\item[\emph{Type 1}:] In this case, the parabolic structure of $E_x$ is a trivial filtration with weight $0$ or $\frac{1}{2}$. Thus the corresponding parabolic line bundle $L_0$ is unique. This is enough to provide the statement of the proposition in this case. We would like to describe next the orthogonal structure of $E$ around the point $x$. The parabolic structure of $E_x$ is given by a trivial filtration; let $E$ be the corresponding $V$-bundle. It is easy to check that the orthogonal structure does not depend on the action of $\mathbb{Z}_2$ on the fiber $E_x$. Let $\mathbb{A}=\begin{pmatrix}
a & b \\
c & d \\
\end{pmatrix}$ be the corresponding orthogonal matrix with respect to the orthogonal structure around $x$. Then $\langle t \cdot, t \cdot \rangle=\langle \cdot, \cdot \rangle$, where $t \in \mathbb{Z}_2$ and $t=\begin{pmatrix}
1 & 0 \\
0 & 1 \\
\end{pmatrix}$ or $\begin{pmatrix}
-1 & 0 \\
0 & -1 \\
\end{pmatrix}$ when considered as matrix, and $t \mathbb{A} t^{-1} = \mathbb{A}$. Note that this relation always holds in this case. Therefore the orthogonal structure around the point $x$ does not depend on the $\mathbb{Z}_2$-action in this case. In conclusion, the relation $L_0^2=\mathcal{O}$ gives a unique orthogonal structure on $E$.
\item[\emph{Type 2}:] In this case, the parabolic structure of $E_x$ is a non-trivial filtration as follows:
\begin{align*}
E=\underbrace{{{E}_{{{x}_{i}},1}}}_{0} \supseteq \underbrace{{{E}_{{{x}_{i}},2}}}_{\frac{1}{2}}\supseteq \left\{ 0 \right\},
\end{align*}
where the numbers 0 and $\frac{1}{2}$ describe the weights. As was discussed in Example \ref{302}, the parabolic line bundle $L_0$ is not uniquely determined; in fact, there are two choices for $L_0$. At the same time, if we consider the orthogonal structure of $E$ around the point $x$, it is easy to find that when considering the $\mathbb{Z}_2$-action as matrix, then $t=\begin{pmatrix}
1 & 0 \\
0 & 1 \\
\end{pmatrix}$ or $\begin{pmatrix}
1 & 0 \\
0 & -1 \\
\end{pmatrix}$. Let $\mathbb{A}=\begin{pmatrix}
a & b \\
c & d \\
\end{pmatrix}$ be the corresponding orthogonal matrix around $x$. Then, similarly, we find there two choices for $\mathbb{A}$, namely
$$\mathbb{A}=\begin{pmatrix}
1 & 0 \\
0 & 1 \\
\end{pmatrix} \text{ or } \begin{pmatrix}
1 & 0 \\
0 & -1 \\
\end{pmatrix}.$$ The latter implies that there are two possible orthogonal structures around the point $x$. Based on the above discussion, we can construct a one-to-one correspondence between parabolic $\text{GL}(2,\mathbb{R})$-Higgs bundles and parabolic line bundles $L_0 \in \text{Pic}_V(M_\eta)$ such that $L_0 ^2 = \mathcal{O}$ in the \emph{Type 2} case.
\end{enumerate}

In conclusion, if $L_0$ is a $2$-torsion point in $\Lambda_{M_\eta}$, then $L_0^2 = \mathcal{O}$. The orthogonal structure of $E$ can be determined based on the parabolic structure of $L_0$ and the condition $L_0^2 = \mathcal{O}$. Moreover, $\Phi$ is symmetric with respect to this orthogonal structure for the same reason as in the non-parabolic case (see Proposition 5.5 in \cite{BarScha}). Conversely, if $(E,\Phi)$ is a $\text{GL}(2,\mathbb{R})$-Higgs bundle, then the orthogonal structure of $E$ gives an isomorphism $(E,\Phi) \cong (E^*, \Phi^t)$, which implies the isomorphism $L_0 \cong L_0^*$ for the corresponding parabolic line bundle $L_0$. Note that in this case the correspondence of $(E,\Phi)$ and $L_0$ is one-to-one. Thus we get a unique element $L_0$ in $\Lambda_{M_\eta}[2]$.
\end{proof}

\begin{rem}\label{702}
In the proof of the above proposition, we only claim that there is a one-to-one correspondence between the line $V$-bundle $L_0$ and the parabolic $\text{GL}(2,\mathbb{R})$-Higgs bundle $(E,\Phi)$. We do not know how to interpret the correspondence between the orthogonal structure of $(E,\Phi)$ and a parabolic structure of $L_0$ in terms of topological invariants. In the non-parabolic case, note that D. Baraglia and L. Schaposnik in \cite{BarScha} as well as N. Hitchin in \cite{Hit2016} gave an interpretation of the correspondence in terms of topological invariants.
\end{rem}

The description for the $\text{PGL}\left( 2,\mathbb{R}\right)$-moduli space is induced to the previous case, according to the definition of a parabolic $\text{PGL}\left( 2,\mathbb{R}\right)$-Higgs bundle:
\begin{prop}\label{703}
With the same notation as above, the fiber of the parabolic Hitchin map $${{\mathsf{\mathcal{M}}}_{par}}\left( \text{PGL}\left( 2,\mathbb{R} \right) \right)\to \mathsf{\mathcal{H}}$$ with respect to a point $a_0$ in the regular locus $H_{reg}$ is the space $$\left(\Lambda^{0}_{M_\eta}[2]\bigoplus\Lambda^{\frac{1}{2}}_{M_\eta}[2]\right) / \pi^* \Lambda_M[2],$$
where $\Lambda^{0}_{M_\eta}[2] \cong \Lambda^{\frac{1}{2}}_{M_\eta}[2] \cong \Lambda_{M_\eta}[2]$ and the superscript $0$ and $\frac{1}{2}$ correspond to the parabolic degree.
\end{prop}

\begin{proof}
Compared to the case when $G=\text{GL}(2,\mathbb{R})$, there are two differences to note:
\begin{enumerate}
\item In the $\text{GL}(2,\mathbb{R})$ case, the parabolic Higgs bundle $(E,\Phi)$ corresponds to a line bundle $L_0$ such that $L_0^2 \cong \mathcal{O}$, i.e. $L_0 \in \text{Jac}_V(M_\eta)$, where $\text{Jac}_V(M_\eta)$ is the subvariety of $\text{Pic}_V(M_\eta)$ consisting of parabolic line bundles with parabolic degree $0$. In the $\text{PGL}(2,\mathbb{R})$ case, the parabolic Higgs bundle $(E,\Phi,A)$ corresponds to a parabolic line bundle $L_0 \in \text{Pic}_V(M_\eta)$ such that $L_0^2 \cong \pi^* A$.
\item In the $\text{PGL}(2,\mathbb{R})$ case, there is an equivalence relation that $(E,\Phi,A) \cong (E'\otimes B, \Phi' \otimes \text{Id}_B, A \otimes B^2)$. In other words, the corresponding parabolic line bundles $L_0$ and $L'_0$ are equivalent, if there is a parabolic line bundle $B \in \text{Pic}_V(M)$ such that $L_0 \cong L'_0 \otimes \pi^{*} B$.
\end{enumerate}
The latter condition says that any parabolic line bundle $L_0$ is equivalent to a parabolic line bundle with parabolic degree $0$ or $\frac{1}{2}$. Applying Proposition \ref{702} provides now the desired statement.
\end{proof}

\begin{prop}\label{704}
\begin{enumerate}
\item The number of orbits under the monodromy action on the Hitchin fibration ${{\mathsf{\mathcal{M}}}_{par}}\left( \text{GL}\left( 2,\mathbb{R} \right) \right)\to \mathsf{\mathcal{H}}$ is $2^s(2^{2g+s-1}-1)+2^{s}\cdot \frac{l}{2}+2^{2g+s-1}$, where $l=2g-2+s$.
\item The number of orbits under the monodromy action on the Hitchin fibration $\mathcal{M}_{par}(\text{PGL}(2,\mathbb{R}))\to \mathsf{\mathcal{H}}$ is $2^{2g+s}+2^s(l-1)$, where $l=2g-2+s$.
\end{enumerate}
\end{prop}

\begin{proof}
In the $\text{GL}(2,\mathbb{R})$-case, recall that the lattice $\Lambda_{M_\eta}$ has the following decomposition from Proposition \ref{509} and Remark \ref{510}:
\begin{align*}
\Lambda_{M_\eta}[2] \cong \left(\bigoplus_{i=1}^{s-1} \mathbb{Z}_2 \right) \oplus \Lambda_{X}[2] \oplus (\mathbb{Z}_2 B)^{ev} \oplus \Lambda_M[2].
\end{align*}
From \S 5.3 the monodromy matrix with respect to this decomposition can be written as
\begin{align*}
A_{ij}^x=\begin{pmatrix}
I_{s-1} & 0 & 0 & 0 \\
0 & I_{2g} & L_{ij}^x & S^x \\
0 & 0 & I & (L_{ij}^x)^t  \\
0 & 0 & 0 & I_{2g+s-1} \\
\end{pmatrix}, \quad
s_{ij}=\begin{pmatrix}
I_{s-1} & 0 & 0 & 0 \\
0 & I_{2g} & 0 & 0 \\
0 & 0 & \sigma_{ij} & 0 \\
0 & 0 & 0 & I_{2g+s-1} \\
\end{pmatrix}.
\end{align*}
We shall discuss the orbits in $\Lambda_{M_\eta}[2]$ with respect to this monodromy action.

The ``maximal" orbits are elements of the form $(w,x,0,0)$, where $w \in \left(\bigoplus_{i=1}^{s-1} \mathbb{Z}_2 \right)$, $x \in \Lambda_X[2]$. Note that each element $(w,x,0,0)$ is a separate orbit under the monodromy action. Thus we have $2^{2g+s-1}$ many ``maximal" orbits.

For the ``non-maximal" case, if $z=0$ and $y \neq 0$, then elements are of the form $(w,x,y,0)$. Note that we can kill $x$ under the monodromy action. Thus the elements can be reduced to the form $(w,0,y,0)$. Note also that the monodromy action includes a permutation action on $(\mathbb{Z}_2 B)^{ev}$, thus there are $2^{s-1} \cdot l$ many orbits.

Now if $z \neq 0$, then we can use the monodromy action to cancel $y$. Thus any element $(w,x,y,z)$ in this case can be reduced to an element $(w,x,0,z)$. With the same discussion as in Theorem 6.8 in \cite{BarScha}, we distinguish two cases: $\langle x,z \rangle=0$ and $\langle x,z \rangle =1$. When $z \neq 0$, we can use the monodromy action to cancel $x$ in the $4$-tuple $(w,x,0,z)$. The element $(w,x,0,z)$ is equivalent to $(w,0,0,z)$. Thus each case gives us $2^{s-1}(2^{2g+s-1}-1)$ many orbits.

In conclusion, the total number of orbits in this case is $$\underbrace{2^{2g+s-1}}_{maximal}+\underbrace{2^{s-1} \cdot l}_{y \neq 0, z=0} + \underbrace{2 \cdot 2^{s-1}(2^{2g+s-1}-1)}_{z \neq 0} = 2^s(2^{2g+s-1}-1)+2^{s}\cdot \frac{l}{2}+2^{2g+s-1}.$$

In the $\text{PGL}(2,\mathbb{R})$ case, we consider the following decomposition of $\Lambda^0_{M_\eta}[2] / (\pi^* \Lambda_{M}[2])$:
\begin{align*}
\Lambda^0_{M_\eta}[2] / (\pi^* \Lambda_{M}[2]) \cong (\mathbb{Z}_2 B)^{ev} \oplus \Lambda_M[2] \cong \underbrace{(\mathbb{Z}_2 B)^{ev}}_{y} \oplus \underbrace{\Lambda_X[2]}_{z_1} \oplus \underbrace{\left(\bigoplus_{i=1}^{s-1} \mathbb{Z}_2 \right)}_{z_2},
\end{align*}
and the decomposition of $\Lambda^{\frac{1}{2}}_{M_\eta}[2] / (\pi^* \Lambda_{M}[2])$ can be considered similarly. When $z_1=0$, the elements $(y,0,z_2)$ describe possible orbits. Fixing $z_2$, under the action of permutation group on $(\mathbb{Z}_2 B)^{ev}$, there are $l-1$ many orbits. Thus the total number of orbits in this case is $2^{s-1}(l-1)$. When $z_1 \neq 0$, any element $(y,z_1,z_2)$ can be reduced to $(0,z_1,z_2)$. Thus we have $2^{2g+s-1}$ many orbits. In conclusion, we have $2^{2g+s-1}+2^{s-1}(l-1)$ many orbits for the lattice $\Lambda^0_{M_\eta}[2] / (\pi^* \Lambda_{M}[2])$. The same discussion holds for $\Lambda^{\frac{1}{2}}_{M_\eta}[2] / (\pi^* \Lambda_{M}[2])$ and the total number of orbits is $2^{2g+s}+2^{s}(l-1)$.
\end{proof}

Putting together Propositions \ref{402}, \ref{404}, \ref{603} and \ref{703} describing the minimum and maximum number of connected components respectively, we obtain the following exact component counts:

\begin{thm}\label{705}Let $X$ be a smooth Riemann surface of genus $g$ and let $D$ be a reduced effective divisor of $s$ many points on $X$.
The number of connected components of the moduli space ${{\mathsf{\mathcal{M}}}_{par}}\left( \text{GL}\left( 2,\mathbb{R} \right) \right)$ of polystable parabolic $\text{GL}(2,\mathbb{R})$-Higgs bundles over the pair $(X, D)$ is $2^{s}(2^{2g+s-1}-1)+ 2^{s} \cdot (g-1+\frac{s}{2})+ 2^{2g+s-1}$.
\end{thm}

\begin{thm}\label{706} Let $X$ be a smooth Riemann surface of genus $g$ and let $D$ be a reduced effective divisor of $s$ many points on $X$.
The number of connected components of the moduli space ${{\mathsf{\mathcal{M}}}_{par}}\left( \text{PGL}\left( 2,\mathbb{R} \right) \right)$ of polystable parabolic $\text{PGL}(2,\mathbb{R})$-Higgs bundles over the pair $(X, D)$ is $2^{2g+s}+2^s(2g-3+s)$.
\end{thm}

\begin{rem}\label{707}
By Theorems \ref{604} and \ref{706}, the number of connected components of the parabolic moduli space in the case of $\text{SL}(2,\mathbb{R})$ and $\text{PGL}(2,\mathbb{R})$ is the same. Moreover, we notice that the monodromy of the Hitchin fibration for the group $\text{SL}(2,\mathbb{C})$ is dual to that for the group $\text{PGL}(2,\mathbb{C})$ (transpose matrices as seen in Remark \ref{511}). In fact, this relationship is derived from the Langlands duality for the parabolic Hitchin systems for the complexifications $G=\text{SL}(2,\mathbb{C})$ and $^{L}G=\text{PGL}(2,\mathbb{C})$; this property is also apparent in the non-parabolic case (see \cite{Bar}, \cite{BarScha}).
\end{rem}

\appendix
\section{The BNR Correspondence for Deligne-Mumford stacks}
\subsection{$V$-cohomology}
The topological invariants of $V$-bundles come from the $V$-cohomology group. To define the $V$-cohomology, we first define the $V$-surface $M_V$. Recall that an atlas of $M$ can be chosen as a union of $M \backslash \{x_1,...,x_s\}$ and $U_i$ around each puncture $x_i$, where $U_i = D_i/ \mathbb{Z}_{\alpha_i}$, $1 \leq i \leq s$. We define $M_V$ as a union of $M \backslash \{x_1,...,x_s\}$ and $\coprod  D_i \times_{\mathbb{Z}_{\alpha_i}}E \mathbb{Z}_{\alpha_i}$, which can be glued naturally.

\begin{defn}\label{901}
The $V$-cohomology group $H^{*}_V(M)$ is then defined as the cohomology
\begin{align*}
H^{*}_V(M)=H^{*}(M_V).
\end{align*}
\end{defn}
We can interpret this $V$-cohomology in terms of the cohomology of a root stack in the following way. The chart $\phi_i:U_i \to D_i/\mathbb{Z}_{\alpha_i}$ around a point $x_i$  is now understood as a map $\varphi_i:D_i\to D_i$, where $D_i$ is a formal disk around $x_i$ and $\varphi_i: z\mapsto z^{\alpha_i}$. So there is a natural $\mu_{\alpha_i}$-action on $D_i$ by sending $z\mapsto \omega_i z$, with $\omega_i $ is the $\alpha_i$'th root of unity. We can glue the quotient stack $[D_i/\mu_{\alpha_i}]$ to $D_i$ via $\varphi_i$ and this forms a smooth Deligne-Mumford stack $\widetilde{M}$. Denote by $X$ the coarse moduli space of $\widetilde{M}$, which is also the underlying space of $M$ in terms of a $V$-surface or orbifold.

We claim that the $V$-cohomology of $M$ is the same as the singular cohomology of this stack $\tilde{M}$. This can be done applying a Mayer-Vietoris argument. We calculate the cohomology of $M_V$ from the cohomology of the disjoint subspaces $A= M\backslash\{x_1,\ldots, x_s\}$ and $B=\cup_i D_i\times_{\mathbb{Z}_{\alpha_i}} E\mathbb{Z}_{\alpha_i}$. We just need to prove that this is the same as gluing $A=M\backslash\{x_1,\ldots, x_s\}$ and $C=\cup [D_i/\mu_{\alpha_i}]$.
This however comes from the definition of root:
\[H^k([D_i/\mu_{\alpha_i}],\mathbb{Z})=H^k(D_{i}\times B\mathbb{Z}_{\alpha_i},\mathbb{Z})=H^k(D_{i}\times_{\mathbb{Z}_{\alpha_i}} E\mathbb{Z}_{\alpha_i},\mathbb{Z}).\]
Moreover, $A\cap D\times_{\mathbb{Z}_{\alpha_i}} E\mathbb{Z}_{\alpha_i}$ is homotopic to $S^1$, which is in turn homotopic to $A\cap [D/\mu_{\alpha_i}]$. The Mayer-Vietoris sequence for orbifolds on $A\cup B= M_V$ and $A\cup C= \tilde{M}$ provides by the five-lemma that $H^i(M_V,\mathbb{Z})\cong H^i(\tilde{M},\mathbb{Z})$:

\[ \begin{tikzcd}[column sep=tiny]
H^{i-1}(A,\mathbb{Z})\oplus H^{i-1}(B,\mathbb{Z})  \arrow{r}{} \arrow[swap]{d}{\cong} & H^{i-1}(A\cap B,\mathbb{Z})\arrow[swap]{d}{\cong} \arrow{r}{}  & H^i(M_V,\mathbb{Z})\arrow{d}{} \arrow{r}{}  & H^{i}(A,\mathbb{Z})\oplus H^{i}(B,\mathbb{Z})  \arrow{r}{} \arrow[swap]{d}{\cong} & H^{i}(A\cap B,\mathbb{Z})\arrow[swap]{d}{\cong} \\
H^{i-1}(A,\mathbb{Z})\oplus H^{i-1}(C,\mathbb{Z})  \arrow{r}{}  & H^{i-1}(A\cap C,\mathbb{Z}) \arrow{r}{}  & H^i(\tilde{M},\mathbb{Z}) \arrow{r}{}  & H^{i}(A,\mathbb{Z})\oplus H^{i}(C,\mathbb{Z})  \arrow{r}{}  & H^{i}(A\cap C,\mathbb{Z})
\end{tikzcd}
\]

The following theorem is the basic tool in order to calculate the cohomology group $H^{*}(M)$:
\begin{thm}[Theorem 2.2 in \cite{FuSt}]\label{902}
We have the following isomorphism for the first $V$-cohomology group
\begin{align*}
H^1_{V}(M,\mathbb{Z}) \cong H^1(M, \mathbb{Z}).
\end{align*}
\end{thm}

\subsection{Parabolic Bundle vs. Bundle over Stack}
We now use the language of stacks to reinterpret the correspondence between parabolic bundles and $V$-bundles over orbifolds; more details can be found in \cite{Biswas1},  \cite{Borne}, or \cite{GroePos}.

Let $X$ be a smooth compact Riemann surface and $D$ a reduced divisor $x_1+\ldots+x_s$. We consider a parabolic bundle on $(X,D)$. Suppose that all weights in the parabolic bundle are rational, and all weights over $x_a$ have the same denominator $\alpha_a$. Let $\tilde{X}$ be the root stack defined by gluing $\bigcup_{a=1}^s [D_a/\mu_{\alpha_a}]$ to the neighborhood of $x_a$ by the map $\varphi_i: z\mapsto z^{\alpha_a}$, for all $a\in \{1,2,\ldots,s\}$. The following theorem is due to N. Borne \cite{Borne}:
\begin{thm}[Th\'{e}or\`{e}me 3.13 in \cite{Borne}]
There is an equivalence of categories between parabolic bundles over $(X,D)$ such that all weights over $x_a$ are with denominator $\alpha_a$, and vector bundles on the root stack $\tilde{X}$.
\end{thm}
We sketch the idea of the proof here. We first notice that from a parabolic bundle $q_a:E\to E|_{x_a}$ we can recover the corresponding parabolic structure on $x_a$. On the other hand, given a parabolic structure $\{0=E_{r+1}\subset\ldots \subset E_0=E|_{x_a} \}$, by taking the inverse $q_a$ we can get a filtration
\[E(-x_a)= F_a^{r+1}\subset F_a^{r}\subset\ldots \subset F_a^1=E. \]
This implies that a parabolic structure on a bundle is the same as a bundle filtration like above on a neighborhood of each point in the divisor.

Now we review the correspondence to an orbifold bundle. Let $\pi: \tilde{X}\to X$ mapping a stack to its underlying space; then for the line bundle $\mathcal{O}(x_a)$ we know that there is a bundle $\mathcal{L}_a$ such that $\mathcal{L}_a^{\otimes \alpha_a}\cong \pi^* \mathcal{O}(x_a)$. We also denote $\mathcal{L}_a\cong \mathcal{O}(\tilde{x}_a)$ and $\mathcal{O}(\tilde{D})=\prod_{a=1}^s \mathcal{O}(\tilde{x}_a)$ so that these bundles actually correspond to the fractional divisor line bundle in the classical theory of orbifold.

For every orbifold bundle $\tilde{E}$, we can define the parabolic bundle by $E=\pi_* \tilde{E}$, and locally on $x_a$ we get the parabolic structure on the neighborhood by setting $F_a^i=\pi_*(\tilde{E}\otimes \mathcal{L}_a^{-i+1})$. Since we already know that $\pi_*\mathcal{L}_a^{-\alpha_a}=\pi_* \pi^* \mathcal{O}(-x_a)= \mathcal{O}(-x_a)$, this forms the desired filtration.

On the other hand, if we have a parabolic bundle with $E_a^i=\ker(E|_{x_a}\to E/F_i^a)$, we can construct an orbifold bundle by gluing together $\mathcal{L}_a^i\otimes \pi^* E_a^i$ in a nice way; we refer to N. Borne's article \cite{Borne} for more details on this step. The result by Borne is that these two maps are actually inverse to each other and thus define an equivalence between categories of parabolic bundles and bundles over a root stack. Similarly, this correspondence can be naturally extended to Higgs bundles.

\begin{thm}[Theorem 4.7 in \cite{BiMaWo}]
There is an equivalence of categories of $K(D)$-twisted Higgs bundles on $\widetilde{X}$ and parabolic Higgs bundles on $(X,D)$.
\end{thm}

For the case of strongly parabolic Higgs bundles, there is a similar correspondence:
\begin{thm}[Theorem 4.7 in \cite{BiMaWo}]
There is an equivalence of categories of Higgs bundles on $\widetilde{X}$ and strongly parabolic Higgs bundles on $(X,D)$.
\end{thm}

\subsection{BNR Correspondence for Orbicurves}

We review the classical theorem describing a correspondence between parabolic Higgs bundles and Higgs bundles over a root stack. Moreover, we shall imply a BNR correspondence on Deligne-Mumford stacks, and see how to reinterpret this correspondence for parabolic Higgs bundles in terms of stacks. The idea is similar to the classical BNR correspondence for smooth curves. We finally refer the reader to \cite{GaPe} for a description of the fibers of the Hitchin map by means of cameral data.

We use the same notation as in the previous subsection. Let $\widetilde{X}$ be a root stack, and let $X$ be its underlying space. We consider the line bundle $K_{\tilde{X}}(\tilde{D})=\pi^*(K_{X}(D))$ over $\widetilde{X}$. In this case, we look at the Hitchin base $H_{st}=\bigoplus_{j=1}^n H^0(\tilde{X}, (K_{\tilde{X}}(\tilde{D}))^j)$ and consider the Hitchin map for stacks

\begin{align*}
h:\mathcal{M}_{Higgs}(\widetilde{X}) \rightarrow H_{st},
\end{align*}
where $\mathcal{M}_{Higgs}(\widetilde{X})$ is the moduli space of Higgs bundles over $\widetilde{X}$. Given an element $\eta=(\eta_1,...,\eta_n) \in H_{st}$, we can consider the spectral curve $\widetilde{X}_\eta$ as well. The following theorem is due to M. Groechenig \cite{GroePos}.

\begin{thm}[BNR correspodence for a Deligne-Mumford stack \cite{GroePos}] Let $\widetilde{X}$ be a smooth Deligne-Mumford stack. Let $\mathfrak{L}$ be a fixed locally free sheaf on $\widetilde{X}$. Then there is a one-to-one correspondence between $\mathfrak{L}$-twisted Higgs bundles on $\widetilde{X}$ with spectral data $\eta$ and coherent sheaves on $\widetilde{X}_\eta$, where $\widetilde{X}_\eta$ is the spectral curve associated to $X$. In particular when $\mathfrak{L}=K_{\tilde{X}}(D)$ we get the parabolic Higgs bundles.
\end{thm}
\begin{proof}
For an $\mathfrak{L}$-twisted Higgs bundle $(E,\Phi)$ on $\widetilde{X}$, we can think of the Higgs field as $\Phi: E\to E\otimes \mathfrak{L}$. Now this is the same as $\Phi: E\otimes \mathfrak{L}^{-1}\to E$, which is giving $E$ a structure of $\mathrm{Sym}^{\bullet}(\mathfrak{L}^{-1})$-module. By the Cayley-Hamilton theorem, this action vanishes at the characteristic polynomial $\lambda^n+\eta_1\lambda^{n-1}+\eta_2\lambda^{n-2}+\cdots+\eta_n$, and so $E$ is actually a $\mathrm{Sym}^{\bullet}(L^{-1} )/(\lambda^n+\eta_1\lambda^{n-1}+\eta_2\lambda^{n-2}+\cdots+\eta_n)$-module, and thus a coherent sheaf on $\widetilde{X}_\eta$.
\end{proof}

Therefore the classical BNR correspondence provides the spectral curve, and now the root stack structure remembers the parabolic structure on the Riemann surface, as we have seen earlier in the proof of the parabolic BNR correspondence.

For strongly parabolic Higgs bundles, the Hitchin base is $$\mathcal{H}:=\bigoplus_{j=1}^n H^0(\tilde{X},K_{\tilde{X}}^j)=\bigoplus_{j=1}^n H^0(X, \pi_*(K_{\tilde{X}}^j)).$$
Since $\pi_*(K_{\tilde{X}}^j)$ is exactly the $\mu_{\alpha_a}$-fixed part locally at each point $x_a$ and $\pi_*(L_a^{m})=O(\lfloor\frac{m}{\alpha_a}\rfloor x_a)$, we have
\[\pi_*(K_{\tilde{X}}^j)=K_X^j\otimes \prod_{a=1}^s \pi_*(L_a^{(\alpha_a-1) \cdot j})=K_X^j\otimes \prod_{a=1}^s O((j-1) x_a)=K_X^j ((j-1)D).\]
This is true as long as $j\le \alpha_a$. Thus $H^0(\tilde{X},K_{\tilde{X}}^j)\cong H^0(X, K_X(D)^j\otimes \mathcal{O}_X(-D))$ so that the Hitchin base for the root stack is actually the Hitchin base for parabolic as well as strongly parabolic Higgs bundles, as desired.

\vspace{2mm}
\textbf{Acknowledgments}.
We would like to express our sincerest acknowledgements to the anonymous referee for their valuable remarks which lead to an improvement of this article. We also thank Laura Schaposnik for making useful comments on a previous article of ours, which provided the motivation for this work, as well as Marina Logares and Andr\'{e} Oliveira for very helpful suggestions. G. K. kindly thanks the Labex IRMIA of the Universit\'{e} de Strasbourg for  support during the completion of this project. H. S. is partially supported by GDBABRF (Guangdong Basic and Applied Basic Research Foundation) 2019A1515110961.

\bigskip

\noindent\small{\textsc{Max-Planck-Institut f\"{u}r Mathematik}\\
Vivatsgasse 7, 53111 Bonn, Germany}\\
\emph{E-mail address}:  \texttt{kydonakis@mpim-bonn.mpg.de}

\bigskip
\noindent\small{\textsc{Department of Mathematics, South China University of Technology}\\
381 Wushan Rd, Guangzhou, Guangdong, China}\\
\emph{E-mail address}:  \texttt{hsun71275@scut.edu.cn}

\bigskip
\noindent\small{\textsc{Department of Mathematics, University of Illinois at Urbana-Champaign}\\
1409 W. Green St, Urbana, IL 61801, USA}\\
\emph{E-mail address}: \texttt{lzhao35@illinois.edu}

\vspace{2mm}


\begin{thebibliography}{99}


\bibitem{Bar}
D. Baraglia, Monodromy of the {$SL(n)$} and {$GL(n)$} Hitchin fibrations. \emph{Math. Ann.} \textbf{370} (2018), no. 3-4, 1681-1716.

\bibitem{BarScha}
D. Baraglia and L.P. Schaposnik, Monodromy of rank 2 twisted Hitchin systems and real character varieties. \emph{Trans. Amer. Math. Soc.} \textbf{370} (2018), no. 8, 5491-5534.

\bibitem{BNR}
A. Beauville, M.S. Narasimhan and S. Ramanan, Spectral curves and the generalised theta divisor. \emph{J. Reine Agew. Math.} \textbf{398} (1989), 169-179.

\bibitem{BiGaRi}
O. Biquard, O. Garc\'{i}a-Prada and I. Mundet i Riera, Parabolic Higgs bundles and representations of the fundamental group of a punctured surface into a real group. \emph{Adv. Math.} \textbf{372} (2020), 107305.

\bibitem{Biswas1}
I. Biswas, Parabolic ample bundles. \emph{Math. Ann.} \textbf{307} (1997), 511-529.

\bibitem{Biswas2}
I. Biswas, Parabolic bundles as orbifold bundles. \emph{Duke Math. J.} \textbf{88} (1997), no. 2, 305-325.

\bibitem{BiMaWo}
I. Biswas, S. Majumder and M. L. Wong, Parabolic Higgs bundles and $\Gamma$-Higgs bundles. \emph{J. Aust. Math. Soc.} \textbf{95} (2013), no. 3, 315-328.

\bibitem{Boden}
H.U. Boden, Representations of orbifold groups and parabolic bundles. \emph{Comment. Math. Helv.} \textbf{66} (1991), no. 3, 389-447.

\bibitem{Borne}
N. Borne, Fibr\'{e}s paraboliques et champ des racines. \emph{Int. Math. Res. Not.} IMRN 2007, no. 16, 38pp.

\bibitem{Cope}
D.J. Copeland, Monodromy of the Hitchin map over hyperelliptic curves. \emph{Int. Math. Res. Not.} (2005), no. 29, 1743-1785.

\bibitem{DHM}
M.A. de Cataldo, T. Hausel and L. Migliorini, Topology of Hitchin systems and Hodge theory of character varieties: the case $A_1$. \emph{Ann. of Math. (2)}, \textbf{175} (2012), no. 3, 1329-1407.

\bibitem{FuSt}
M. Furuta and B. Steer, Seifert fibred homology 3-spheres and the Yang-Mills equations on Riemann surfaces with marked points. \emph{Adv. Math.} \textbf{96} (1992), no. 1, 38-102.

\bibitem{GaGoMu}
O. Garc\'{i}a-Prada, P. B. Gothen and V. Mu\~{n}oz, Betti numbers of the moduli space of rank 3 parabolic Higgs bundles. \emph{Mem. Amer. Math. Soc.} \textbf{187} (2007), no. 879, viii+80 pp.

\bibitem{GaLoMu}
O. Garc\'{i}a-Prada, M. Logares and V. Mu\~{n}oz, Moduli spaces of parabolic $\mathrm{U}\left(p,q \right)$-Higgs bundles. \emph{Q. J. Math.} \textbf{60} (2009), no. 2, 183-233.

\bibitem{GaPe}
O. Garc\'{i}a-Prada and A. Pe\'{o}n-Nieto, Higgs bundles, abelian gerbes and cameral data. arXiv:1902.06139.

\bibitem{GoLog}
T. G\'{o}mez and M. Logares, A Torelli theorem for the moduli space of parabolic Higgs bundles. \emph{Adv. Geom.} \textbf{11} (2011), no. 3, 429-444.

\bibitem{GoOl}
P.B. Gothen and A. Oliveira, Topological mirror symmetry for parabolic Higgs bundles. \emph{J. Geom. Phys.} \textbf{137} (2019), 7-34.

\bibitem{GroePos}
M. Groechenig, Moduli of flat connections in positive characteristic. \emph{Math. Res. Let.} \textbf{23} (2016), no. 4, 989-1047.

\bibitem{Hit2016}
N. Hitchin, Higgs bundles and characteristic classes. Arbeitstagung Bonn 2013, 247-264, \emph{Progr. Math.} \textbf{319} Birkh\"{a}user/Springer, Cham, 2016.

\bibitem{KSZ2}
G. Kydonakis, H. Sun and L. Zhao, The Beauville-Narasimhan-Ramanan correspondence for twisted Higgs $V$-bundles and components of parabolic $\text {Sp}(2n,\mathbb {R}) $-Higgs moduli spaces. \emph{to appear in Trans. Amer. Math. Soc.} (2020).

\bibitem{KSZ}
G. Kydonakis, H. Sun and L. Zhao, Topological invariants of parabolic $G$-Higgs bundles. \emph{to appear in Math. Z.} (2020).

\bibitem{LogaBetti}
M. Logares, Betti numbers of parabolic $\text{U}(2,1)$-Higgs bundles moduli spaces. \emph{Geom. Dedicata} \textbf{123} (2006), 187-200.

\bibitem{Loga}
M. Logares, Parabolic $\text{U}(p,q)$-Higgs bundles. Ph. D. thesis, Universidad Aut{\'o}noma de Madrid (2006).

\bibitem{LoMa}
M. Logares and J. Martens, Moduli of parabolic Higgs bundles and Atiyah algebroids. \emph{J. Reine Angew. Math.} \textbf{649} (2010), 89-116.

\bibitem{NaSt}
B. Nasatyr and B. Steer, Orbifold Riemann surfaces and the Yang-Mills-Higgs equations.  \emph{Ann. Scuola Norm. Sup. Pisa Cl. Sci. (4)} \textbf{22} (1995), no. 4, 595-643.

\bibitem{Schap}
L.P. Schaposnik, Monodromy of the $\text{SL}_{2}$ Hitchin fibration. \emph{Int. J. Math.} \textbf{24} (2013), no. 2, 182-203.

\bibitem{Sesh}
C.S. Seshadri, Fibr\'{e}s vectoriels sur les courbes alg\'{e}briques.  \emph{Ast\'{e}risque} \textbf{96}, Soci\'{e}t\'{e} Math\'{e}matique de France, Paris, 1982, 209 pp.

\bibitem{Simp}
C.T. Simpson, Harmonic bundles on noncompact curves. \emph{J. Amer. Math. Soc.} \textbf{3} (1990), no. 3, 713-770.

\bibitem{SWW}
X. Su, B. Wang and X. Wen, Parabolic Hitchin maps and their generic fibers. arXiv:1906.04475.

\bibitem{Walk}
K.C. Walker, Quotient groups of the fundamental groups of certain strata of the moduli space of quadratic differentials. \emph{Geom. Topol.} \textbf{14} (2010), no. 2, 1129-1164.

\bibitem{Yoko1}
K. Yokogawa, Compactification of moduli of parabolic sheaves and moduli of parabolic Higgs sheaves. \emph{J. Math. Kyoto Univ.} \textbf{33} (1993), 451-504.

\bibitem{Yoko2}
K. Yokogawa, Infinitesimal deformation of parabolic Higgs sheaves. \emph{Int. J. Math.} \textbf{6} (1995), 125-148.

\end{thebibliography}
\end{document}